\documentclass{amsart}
\usepackage{amsmath}
  \usepackage{paralist}
  \usepackage{graphics} 
  \usepackage{epsfig} 
\usepackage{graphicx}  \usepackage{epstopdf}
 \usepackage[colorlinks=true]{hyperref}
\hypersetup{urlcolor=blue, citecolor=red}

\newtheorem{thm}{Theorem}[section]
\newtheorem{prop}[thm]{Proposition}
\newtheorem{cor}[thm]{Corollary}

\newtheorem{remarque}[thm]{Remark}
\theoremstyle{defi}
\newtheorem{defi}[thm]{Definition}
\newtheorem{notation}[thm]{Notation}
\newtheorem{condition}[thm]{Condition}
\title[Maximal estimates for the Fokker-Planck-Magnetic operator ] 
      {Maximal estimates for the Fokker-Planck operator with magnetic field}

\author[Zeinab karaki]{}


 \email{zeinab.karaki@univ-nantes.fr}

\begin{document}
\maketitle

\centerline{\scshape Zeinab Karaki}
\medskip
{\footnotesize
 \centerline{Universit\'e de Nantes}
   \centerline{Laboratoire de Mathematiques Jean Leray}
   \centerline{ 2, rue de la Houssini\`ere}
    \centerline{BP 92208 F-44322 Nantes Cedex 3, France}
} 

\medskip

\bigskip


 \begin{abstract}
We consider the Fokker-Planck operator with a strong external magnetic field. We show a maximal type estimate on this operator using a nilpotent approach on vector field polynomial operators and including the notion of representation of a Lie algebra. This estimate makes it possible to give an optimal characterization of the domain of the closure of the considered operator. 
\end{abstract}

\section{Introduction and main results}
\subsection{Introduction}
The Fokker-Planck equation was introduced by Fokker \cite{fokker1914ad}  and Planck
\cite{Planck-FPE-1917}, to describe the evolution of the density of particles under Brownian motion. In recent years, global hypoelliptic estimates have experienced a rebirth
through applications to the kinetic theory of gases. In
this direction many authors have shown maximal estimates
to deduce the compactness of the resolvent of the Fokker-Planck operator and
resolvent estimates to address the issue of return to
the equilibrium. F. H\'erau and F. Nier in \cite{herau2004isotropic} have highlighted the
links between the Fokker-Planck operator with a confining potential and
the associated  Witten Laplacian. Then, in the book of B. Helffer
and F. Nier \cite{nier2005hypoelliptic}, this work has been extended and explained in a general way,
and we refer more specifically to Chapter 9 for a demonstration of
the maximal estimate.

In this article, we continue the study of the model case of the
Fokker-Planck operator with an external magnetic field $B_e$, started in 
\cite{ZK}, and we establish a maximal-type estimate for this model, giving a
characterization of the domain of its closed extension. 
\subsection{Statement of the result}
For $d=2$ or $3$, we consider the Fokker-Planck operator $K$ with an external magnetic field   $B_e$ defined on 
$\mathbb T^d:=\mathbb R^d/ \mathbb Z^d$ with values in $\mathbb{R}^{d(d-1)/2}$ such that
\begin{align}\label{def_K}
K=v\cdot \nabla_x -(v\wedge B_e )\cdot \nabla_v -\Delta_v + v^2/4 -d/2, 
\end{align}

where $v\in\mathbb{R}^{d}$ represents the velocity, $x \in \mathbb{T}^d$ represents the space variable and $t> 0$ is the time.
 In the previous definition of our
operator, we use $(v \wedge B_e ) \cdot \nabla_v$ to mean
\begin{equation*}
(v\wedge B_e)\cdot \nabla_{v}=\begin{cases}
b(x)\,(v_1 \partial_{v_{2}} -v_2 \partial_{v_{1}}) \quad \quad &\text{ if } d=2 \\~\\
b_1(x)(v_2 \partial_{v_{3}} -v_3 \partial_{v_{2}})+b_2(x) (v_3 \partial_{v_{1}} -v_1 \partial_{v_{3}})\\
+ b_3(x)(v_1 \partial_{v_{2}} -v_2 \partial_{v_{1}})  \quad \qquad
&\text{ if } d=3.
\end{cases}
\end{equation*}
 The operator $K$ is considered as an unbounded operator on the Hilbert space $H=L^2 (\mathbb{T}^{d}\times \mathbb{R}^{d})$ whose domain is the Schwartz space 
  $D(K)=\mathcal{S}(\mathbb T^d\times \mathbb{R}^d)$.  We denote by
\begin{enumerate}
\item[$\bullet$] $K_{\min}$ the minimal extension of $K$  where $D(K_{\min})$ is the closure of $D(K)$ with respect to the graph norm on $H\times H$.
 \item[$ \bullet $] $K_{\max}$ is the maximal extension of $K$ whose domain $D(K_{\max})$ is given by
$$ D(K_{\max}) =\{u\in L^2 (\mathbb{T}^{d}\times \mathbb{R}^{d}) \,/\, Ku\in L^2 (\mathbb{T}^{d}\times \mathbb{R}^{d})\}.$$
\end{enumerate}
From now on, we use the notation $\bf{K}$ for the operator $K_{\min}$.

 The existence of a strongly
continuous semi-group associated to the operator $\bf K$ is shown in  \cite{ZK} when the magnetic field is regular. We improve this result by considering a much lower regularity. In order to obtain the maximal accretivity, we are led to substitute
the hypoellipticity argument by a regularity argument for the operators with coefficients in $ L^\infty $, which will be combined with more classical results of Rothschild-Stein in \cite{rothschild1979criterion} for H\"ormander operators of type-2 (see \cite{hormander1967hypoelliptic} for more details of this subject).
\begin{thm}\label{prop 2}
If $B_{e}\in L^{\infty}(\mathbb{T}^{d},\mathbb{R}^{d(d-1)/2})$ , then $\bf K$  is maximally accretive.
\end{thm}
This implies that the domain of the operator $\bf K$ has the following property:
\begin{align}
\label{eq:1} 
D({\bf K})= D(K_{\max})\,.
\end{align}

  In this work, we are interested in specifying the domain of the operator $ \bf K
  $ introduced in \eqref{eq:1}. To accomplish this,  we will obtain maximal estimates for $ \bf K $, using techniques developed initially for the study of hypoellipticity of invariant operators on nilpotent groups. Before we state our main result, we establish some notation.
\begin{notation}~\
\begin{itemize}
\item $B^{2}(\mathbb R^d)$ (or $B^2_v$ to indicate the name of the variables) denotes the space
 $$B^2(\mathbb R^d) :=\{u\in L^{2}(\mathbb{R}^{d})\,/\, \forall (\alpha, \beta )\in \mathbb{N}^{2d},\, \vert \alpha \vert +\vert \beta \vert \leq 2 \,,\, v^{\alpha}\,\partial^{\beta}_{v}\,u\in L^{2}(\mathbb{R}^{d}) \}, $$
 which is equipped with its natural Hilbertian norm
\item $ \tilde{B}^2 (\mathbb{T}^d \times \mathbb{R}^d) $ is the space $ L^{2}  (\mathbb{T}^d_{x}, B^2_v(\mathbb{R}^d)) $
with the following Hilbertian norm:
\begin{align*}
 \tilde{B}^2 (\mathbb T^d\times \mathbb R^d) \ni u\longrightarrow \Vert u \Vert_{\tilde{B}^2}= \sqrt{\sum\limits_{
\vert \alpha \vert +\vert \beta \vert \leq 2} \, \left\Vert  v^{ \alpha}\, \partial^{\beta}_{v} u \right\Vert^{2}}.
\end{align*}
where $\Vert .\Vert$ is the $L^2 (\Bbb T^d \times \Bbb R^d)$ norm .
\item $\mathrm{Lipsch}(\mathbb{T}^{d})$ is the space of Lipschitizian functions from $\mathbb{T}^{d}$ with values in $\mathbb R^{d(d-1)/2}$, equipped 
 with the following norm:
$$ \Vert u \Vert_{\mathrm{Lipsch}(\mathbb{T}^{d})}=\Vert u\Vert_{L^{\infty}(\mathbb{T}^{d},\mathbb R^{d(d-1)/2})} + \sup\limits_{x,y\in \mathbb{T}^{d},
 x\neq y} \, \frac{\vert u (x)-u (y)\vert}{ d(x,y)}\,,$$ 
where $d$ is the natural distance in $\mathbb T^d$ and $\vert \cdot\vert$ is the Euclidean norm in $\mathbb R^{d(d-1)/2}$.
 \end{itemize}
\end{notation}
We can now state the main theorem of this article:
\begin{thm}\label{hypoelli0}
Let $ d=2 \text{ or } 3$. We assume that $ B_e \in \mathrm{Lipsch}(\mathbb{T}^{d})$. Then   for any $C_{1}>0$, there exists some $C>0$ such that for all $B_{e}$ with
 $\Vert B_{e}  \Vert_{\mathrm{Lipsch}(\mathbb{T}^{d})}\leq C_{1},$
 and for all $u \in \mathcal{S}(\mathbb{T}^{d}\times \mathbb{R}^{d})$, the operator $K$ satisfies the following maximal estimate:
\begin{align}
\Vert (v\cdot \nabla_{x} -(v\wedge B_e )\cdot \nabla_v )u\Vert+\Vert u \Vert_{\tilde{B}^{2}} \leq C(\,\Vert K u\Vert +\Vert u\Vert\, ) ,
\label{hypomax}
\end{align}
where $\Vert .\Vert$ is the $L^2 $ norm.
\end{thm}
Given the density of  $\mathcal{S}(\mathbb{T}^d \times \mathbb{R}^d)$  in the domain of  $\bf K$, we obtain the following characterization of the domain of $ \bf K $:
\begin{cor}
\begin{align*}
D({ \bf K} )=\{u\in \tilde B ^2(\mathbb{T}^{d}\times\mathbb{R}^{d}) \, /\, \left(v\cdot\nabla_{x} -(v\wedge B_e )\cdot \nabla_{v}\right) u \in L^{2}(\mathbb{T}^{d}\times\mathbb{R}^{d}) \}.
\end{align*}
\end{cor}
~\\
\textbf{Organization of the article:}\\
 In the next section, we recall the notion of maximal hypoellipticity, and more specifically, we give some results of maximal hypoellipticity for polynomial operators of vector fields with the nilpotent approach. Then, in the sections 3 and 4, we use these techniques to prove the main result Theorem \ref{hypoelli0}, beginning with $ d = 2 $ and continuing with $ d = $ 3. Finally, in the appendix we give the proof of Theorem \ref{prop 2}\,.
\section{Review of maximal hypoellipticity in the nilpotent approach}
\subsection{Maximal hypoellipticity for polynomial operators of vector fields}
We are interested in the polynomial operators of vector fields. We consider $p+q$ real  $ C^{\infty} $ vector fields $(X_1, ..., X_p, Y_{1}, ..., Y_{q}) $  on an open set $ \Omega \subset \mathbb{R}^{d} $.
~\par Let $ P (z,\zeta_1, ..., \zeta_{p + q}) $ be a non-commutative polynomial of degree $ m $ in $ p + q $ variables, with $ C^{\infty} $ coefficients on $ \Omega $, and let $ \mathcal{P} $ be the differential operator
\begin{align}
\mathcal{P}=P(z,Z_1,...,Z_{p+q})=\sum_{\vert \alpha \vert \leq m} \,  a_{\alpha}(z)\,Z^{\alpha}, \quad \forall \alpha \in  \{1,...,p+q\}^{k}\,,
\label{po}
\end{align}
where, for  $ \ell \in \{1, ..., p + q \} $,  the vector field $ Z_{\ell} $,  is defined by
\begin{align*}
 &Z_{\ell}=X_{\ell} , \quad \forall \ell= 1,...,p,\\
 &Z_{\ell}=Y_{\ell-p},\quad \forall \ell= p+1,...,p+q,
\end{align*}
and where, for $\alpha=(\alpha_{1},..,\alpha_{k})\in \{1,...,p+q\}^{k}$,
$$\vert \alpha \vert =\sum_{j=1}^{k}\, d(\alpha_{j})\, \text{ with } \, d(\alpha_{j})=\begin{cases} 1\quad \text{  if } \alpha_{j}\in \{1,..,p\}\\
2\quad \text{  if } \alpha_{j}\in \{ p+1,..,p+q\}.
\end{cases} $$

It is further assumed that the vector fields $ Z_j $ with $ j \in \{1, ..., p + q \} $ satisfy the H\"ormander condition in $\Omega$: 
\begin{condition}\label{ho}
 There exists an integer $ r $ such that the vector space spanned by the iterated brackets of the vector fields $ Z_{j} $ of length less than or equal to $ r $, in each point $ z $ of $ \Omega, $ is  all of $ T_{z} \Omega $. 
\end{condition}
When $ q = 0 $ and the vector fields $Z_j$ satisfy Condition \ref{ho}, the operator $ \mathcal{P} $ is called a differential operator of type-$ 1 $. Thus the H\"ormander operator $\sum_{j=1}^p  X_j^2$ (case $q=0$) is called a ``\textbf{type-1 H\"ormander operator}".\\
When $q=1$ and the vector fields $Z_j$ satisfy Condition \ref{ho}, the operator $\mathcal{P}$ is called a differential operator of type-2.
 The operator also studied by H\"ormander $\sum_{j=1}^p X_j^2 + Y_1$ (case $ q = 1 $) is called a ``\textbf{type-2 H\"ormander operator}".

Now we introduce the following definition.
\begin{defi}
Let $ m \in \mathbb{N}^* $. The operator $ \mathcal{P} $ is maximal hypoelliptic at a point $ z $ of $ \Omega $, if there is a neighborhood $ \omega $ of $ z $, and a constant $ C >0 $ such that 
\begin{align}
\Vert u\Vert_{\mathcal{H}^{m}(\omega)}^{2} \leq C[\Vert \mathcal{P}u\Vert_{L^{2}(\omega)}^{2}+\Vert u \Vert_{L^{2}(\omega )}^{2}],\quad \forall u\in C^{\infty}_{0}(\omega), 
\end{align}
where $ \Vert \cdot\Vert_{\mathcal{H}^{m} (\omega)} $ denotes the standard norm whose square is defined by
\begin{align}\label{espace_bis}
u\mapsto \Vert u\Vert^{2}_{\mathcal{H}^{m}(\omega)} =\sum_{\vert \alpha \vert \leq m} \, \Vert Z_{\alpha_1}...Z_{\alpha_{k}} u\Vert_{L^{2}(\omega)}^{2},\quad\forall \alpha \in \{1,...,p\}^{m}.
\end{align}
\end{defi}
It has been shown by Helffer-Nourrigat in \cite{helffer1978hypoellipticite} that if the H\"ormander condition is verified at $ z $, then the maximal hypoellipticity at $ z $ implies that $ \mathcal{P}$ is hypoelliptic in a neighborhood of $ z $, which justifies the terminology. 

   It has also been shown by Rothschild-Stein in \cite{rothschild1976hypoelliptic} that 
if the homogeneous operator associated to $\mathcal{P} $ ({\it i.e.} $\sum_{\vert \alpha \vert=m}\,a_\alpha (z)\,Z^\alpha$) is hypoelliptic at a point $z_0$, in the sense introduced by L.~Schwartz, then $ \mathcal{P} $ is maximal hypoelliptic in a neighborhood of $ z_0 $.

\subsection{Nilpotent and graded Lie algebras}
We refer the reader to \cite[Chapter 2]{helffer1980hypoellipticite} for more details on this subject.
\begin{defi}
We say that a Lie algebra $ \mathcal{G} $ is graded nilpotent of rank $ r $, if it admits a decomposition of the form
\begin{align*}
&\mathcal{G}=\mathcal{G}_{1}\oplus...\oplus\mathcal{G}_{r}\,, \\& [\mathcal{G}_{i},\mathcal{G}_{j}]\subset \mathcal{G}_{i+j} \quad \text{ if } i+j\leq r \quad  \text{ and }\quad[\mathcal{G}_{i},\mathcal{G}_{j}]=0 \quad \text{ if } i+j>r\,.
\end{align*}
\end{defi}
\begin{defi}
Let $ \mathcal {G} = \mathcal{G}_{1} \oplus ... \oplus \mathcal{G}_{r} $ be a graded nilpotent Lie algebra with rank $ r $.
\begin{enumerate}
\item We say $\mathcal{G}$ is stratified of type ~ $ 1 $ (or simply stratified) if it is generated by $ \mathcal{G}_{1} $.
\item We say $\mathcal{G}$ is stratified of type ~$ 2 $ if it is generated by $ \mathcal {G}_{1} \oplus \mathcal{G}_{2} $.
\end{enumerate}
\end{defi}
From now on, $\mathcal{G}$ will always refer to a graded nilpotent Lie algebra of rank $r$.
\subsection{Representations  theory on Lie algebras }
Among the representations, the irreducible unitary representations play a crucial role. Kirillov's theory allows us to associate to every element of the dual $ \mathcal{G}^{*} $ of $ \mathcal{G} $ an irreducible representation.
Moreover, this theory says that any irreducible unitary representation can be represented in this way.

 To be more precise, we give a definition of an induced representation. The starting point is a subalgebra $ \mathcal{H} \subset \mathcal{G} $ and a linear form $ \ell $ on $ \mathcal{G} $ such that $$ \ell ([\mathcal{H} , \mathcal{H}]) = 0 .$$ We will then associate a representation $ \pi_{\ell, \mathcal{H}}.$ of the group  $ G:=\exp (\mathcal{G}) $ in $ V_\pi: = L^{2} (\mathbb{R}^{k (\pi)} ) $ which is uniquely defined  modulo unitary conjugation, where $ k (\pi) $ is the codimension of $ \mathcal{H} $ in $ \mathcal{G} $. For this construction and using the nilpotent character, we can find $ k = k (\pi) $ linearly independent vectors $ e_1, .., e_k $ such that any $ a\in \mathcal{G} $ can be written in the form:
\begin{align}
g:=\exp(a)=h \, \exp(s_k e_k) \, ... \,\exp(s_1 e_1) \label{fo 1}\,,
\end{align} 
and such that, if 
$$ \mathcal{A}_{j}= \mathcal{H}\oplus\mathbb{R}e_1\oplus..\oplus \mathbb{R}e_{k-j+1} \,,$$
then $ \mathcal{A}_{j-1} $ is ideal of codimension one in $ \mathcal{A}_{j} $.

With this construction, we can obtain that $ g \mapsto (s, h) $ is a global diffeomorphism from $ \mathcal{G} $ to $ \mathbb{R}^{k} \times \mathcal{H}
$. The induced representation is given by
$$ \left(\pi_{\ell ,  \mathcal{H}} ( \exp a ) f\right)(t) = \exp i\, \langle \ell , h(t,a)\rangle\, f(\sigma (t,a)),$$
where $ h (t, a) $ and $ \sigma (t, a) $ are defined by the following formula:
$$ \exp t_k e_k .... \exp t_1 e_1\, \exp a=\exp (h(t,a))\exp \sigma_k (t,a) e_k...\exp \sigma_1 (t,a) e_1\,. $$

We also note $ \pi_{\ell, \mathcal{H}} $ the representation of the associated Lie algebra defined by
\begin{align}\label{def 8}
  \pi_{\ell , \mathcal{H}} (a)u =\frac{d}{ds} \,\left( \pi_{\ell , \mathcal{H}} (e^{sa})u \right)
  _{|_{s=0}}\,, 
  \end{align} 
where the representation $ \pi_{\ell, \mathcal{H}} $ can be defined on the set of $ u \in V_{\pi} $ such that the mapping $ s \mapsto \pi_{\ell, \mathcal{H}} (e^{as}) u $ is of class $ C^{1} $.
We will actually work on the space $ \mathcal S_ \pi $ of the $ C^{\infty} $  representations $ \pi_{\ell, \mathcal{H}} $. \\
 
  More explicitly, we have  
  $$\pi_{\ell , \mathcal{H}}(a)=i\langle \ell ,h'(t,a) \rangle  + \sum_{j=1}^{k} \, \sigma'_{j} (t,a) \, \partial_{t_{j}},$$
  where $h' $ and $\sigma' $ designate
  \begin{align*}
h'(t,a)&:=\frac{d}{ds}\, \left(h(t,sa) \right)_{|_{s=0}},\\
\sigma'(t,a)&:=\frac{d}{ds}\, \left(\sigma(t,sa) \right)_{|_{s=0}}.
\end{align*}
In addition, $ \sigma $ has the following structure:
\begin{equation}\label{sigma}
\sigma_{j}(t_{j},....,t_1,a)=t_j + \psi_{j}(t_{j-1},..,t_1,a)\,,
\end{equation}
where $ \psi_{j} $ are polynomials on $ \mathbb{R}^{k} $, depending only on the given variables, with real coefficients.

We know from Kirillov's theory that, in the nilpotent case, the irreducible representations are associated with elements of $ \mathcal{G}^{*}$ and that, when $ \pi $ is irreducible, the space $ V_{\pi} $ identifies with $ L^{2} (\mathbb{R}^{k (\pi)}) $ where $ k (\pi) $ is a integer with $  L^2 (\mathbb{R}^{0}) = \mathbb{C} $ by convention. We denote by $ \widehat{G} $ the set of irreducible representations of the  simply connected  group $ G : = \exp \mathcal{G} $ associated to $ \mathcal{G} $. It is also important to note that in the case of an irreducible representation, $ \mathcal{S}_\pi $ identifies with the Schwartz space $ \mathcal{S} (\mathbb{R}^{k (\pi)}) $. \\
Returning to the induced representations $ \pi_{\ell, \mathcal{H}} $, two particular cases will interest us.

  When $ \ell = 0 $, we obtain the standard extension of the trivial representation of the $ H $ subgroup of $ G $. We can consider this as a representation on $ L^{2}(G / H) $. An interesting problem (which is solved in \cite{helffer1980hypoellipticite}) is to characterize the maximal hypoellipticity of $ \pi_{0, \mathcal{H}} (P) $ for $ P \in \mathcal{U}_{m} (\mathcal{G}) $ (elements of $\mathcal{U}(\mathcal{G}$) with degree $m$).
~ \par The second case is when the subalgebra $ \mathcal{H}\subset \mathcal{G}$ is of maximal dimension, for a fixed form $ \ell \in \mathcal{G}^{*} $, with the above property. In this case, we can show that the representation is irreducible. Moreover one can thus construct all the irreducible representations (up o unitary equivalence). Starting this time with an element $ \ell \in \mathcal{G}^{*} $, we can construct a maximal subalgebra $ V_{\ell} $ such that $ \ell ([V_{\ell} , V_{\ell}]) = 0 $. We can also show that the codimension $ k (\ell) $  of $ V_{\ell} $ is equal to $ \frac{1}{2} \, \mathrm{rank} \, B_{\ell} $, where $ B_{\ell} $ is the $ 2$-form defined by
$$ \mathcal{G}\times \mathcal{G}\to \ell ([X, Y]). $$
For $ a \in \mathcal{G} $, we define by $ (\mathrm{ad} \, a)^{*} $ the adjoint of $ \mathrm{ad} \, a:b\ni \mathcal{G}\to (\mathrm{ad}\, a)b:=[a,b]$ which is an endomorphism of $ \mathcal G^* $ defined by
$$ (\mathrm{ad} \, a )^{*}\, \ell (b):=\ell ([a,b]).$$
The group  $ G $ then naturally acts on $ \mathcal{G}^{*} $ by
$$ g \mapsto g \ell = \sum_{k = 0}^{r} \, \frac{(-1)^k}{k!} \, (\mathrm{ad} a)^{*k} \, \ell, $$
with $ g = \exp (a) $. \\

This action is called the coadjoint action. Kirillov's theory tells us that if $ \ell $ and $ \tilde{\ell} $ are on the same orbit for the coadjoint action, then the corresponding unitary representations are equivalent. Conversely, two different orbits give two non-equivalent irreducible representations. We can thus identify $ \widehat{G} $ with the set of irreducible representations of $ G $ with the set of $ G $-orbits in $ \mathcal{G}^{*} $:
$$\widehat{G}=\mathcal{G}^{*}/G. $$ 
In the proof of the main theorem, we find a class of representations of a Lie algebra $ \mathcal G $ in the space $ \mathcal S (\mathbb{R}^{k}) $ ($ k \geq 1 $) that have the following form:
\begin{defi}\label{definition10} For all $ X \in \mathcal{G} $, we define the representation $\pi$ as follows
\begin{align}
\label{def representation}
\pi(X)= P_{1}(X)\frac{ \partial}{\partial y_1}+P_{2} (y_1;X)\frac{\partial}{\partial y_2}+...+P_{k}(y_1,..,y_{k-1};X)\frac{\partial}{\partial y_k}+i\, Q(y_1,..,y_k;X),
\end{align}
where $ P_j (y; X) $ and $ Q (y; X) $ are polynomials in $y\in \mathbb{R}^{k} $, depending only on the given variables, with real coefficients, depending linearly on $ X \in \mathcal{G} $ and the linear forms $\{ X \mapsto P_j (0; X) \}_{1 \leq j \leq k }$ are linearly independent in $ \mathcal{G}^{*} $.  
\end{defi}
  We note that the induced representations always have this form. Conversely, it is natural to ask if the representation defined in Definition \ref{definition10} actually an induced representation?

  The positive answer is given by Helffer-Nourrigat in \cite{helffer1980hypoellipticite}.
By requiring, for all $ X \in \mathcal{G} $
\begin{align}
\label{condition 1}
\langle \ell, X\rangle =Q(0;X),
\end{align}
and denoting by $ \mathcal{H} $ the subspace of $ X \in \mathcal{G} $ such that $ P_j (0; X) = 0 $ ($ j = 1, ..., k $ ), these authors prove the following proposition:
 
\begin{prop} \label{prop 2.11}
Under the assumptions above,
\begin{enumerate}
\item[i)] The subspace $ \mathcal{H} $ is a subalgebra of $ \mathcal{G} $, and we have
\begin{align}
\label{condition 2}
\langle \ell, [X, Y] \rangle = 0 \quad \forall X, Y \in \mathcal {H} \,.
\end{align}
\item[ii)] The representation $ \pi $ is unitarily equivalent to $ \pi_{\ell, \mathcal{H}} $. There exists a unitary transform  $T$ such that
$$ \pi_{\ell, \mathcal{H}}(\exp X) Tf=T \pi (\exp X)f, \, \forall X\in \mathcal{G} \,\text{ and } \forall f\in \mathcal{S}(\mathbb{R}^{k}),$$
 where
 \begin{itemize}
\item $ T $ is defined by
\begin{align*}
Tf(t)= e^{i\varphi (t)}\,f(\theta (t)),\, \forall f\in \mathcal{S}(\mathbb{R}^k),  \forall t\in \mathbb{R}^k
\end{align*}
 where $\varphi$ is defined by
$$\varphi (t)=\langle  \ell, \gamma (t)\rangle=Q(0;\gamma (t))\,, $$
with $\gamma (t)$ is an element of $\mathcal{G}$ such that
$$\exp\, \gamma (t)=(\exp\, t_k \,X_k)...(\exp\, t_1\, X_1)\,. $$
\item $ \theta (t) = (\theta_{1} (t), ..., \theta_{k} (t)) $ is  global diffeomorphism of $ \mathbb{R}^{k} $ defined by
\begin{align*}
\theta_1 (t)&=t_1\\
\theta_j (t)&=t_j+H_{j}(t_1,..,t_{j-1} )\,,\, \forall j=2,..,k \,,
\end{align*}
\item $ H_j $ are polynomials.
\end{itemize}
\end{enumerate}
In particular $T$ sends $\mathcal S (\mathbb R^k)$ into $\mathcal S(\mathbb R^k)$.
\end{prop}

Note that $ \theta $ is a global diffeomorphism of $ \mathbb{R}^{k} $, whose Jacobian is $ 1 $. We thus pass without any problem from one maximal inequality
  to the another by a change of variables which preserves Lebesgue measure.\\
  
We finish this part by proving a specific property that will be used later.
\begin{prop} \label{prop anti}
Let $ \pi $ be a representation defined by \eqref{def representation}. Then for every $ X \in \mathcal{G} $, the operator $ \pi (X) $ is formally skew-adjoint for the usual scalar product defined on the space $ L^2 (\mathbb{R}^{k}) $, i.e.
 $$ \forall u, v\in \mathcal S(\mathbb{R}^{k}), \quad \langle \pi (X)\,u, v \rangle = -\langle u,\pi (X)\,v \rangle.$$
 \begin{proof}
Let $ u , v \in \mathcal{S}(\mathbb{R}^k) $. For $ X \in \mathcal{G} $, we have
$$\langle \pi (X) u, v\rangle =\int_{\mathbb{R}^{k}}\, \pi(X)\,u(y)\, \overline{v(y)}\,dy , $$
by performing an integration by parts as a function of $ y_j $ with $ j = 1, ..., k $ and using the fact that $ P_{j} (y; X) $ are polynomials with real coefficients that depend only on $ y_1, .., y_ {j-1} $ for all $ j = 1, .., k $, we get
\begin{align*}
&\langle \pi (X) \,u , v\rangle \\
&=-\sum_{j=1}^{k}\,\int_{\mathbb{R}^{k}}\, u(y)\, \overline{P_{j}(y_1,..y_{j-1};X)\,\frac{\partial v}{\partial y_j}(y)}\,dy + \int_{\mathbb{R}^{k}}\, u(y)\,\overline{\left(-i\,Q(y_1,..,y_k;X)v(y)\right)}\,dy
,\end{align*}
since $ Q (y; X) $ is also a polynomial with real coefficients. Then by reusing the definition \eqref{def representation} of the representation $ \pi (X) $, we obtain
\begin{align*}
\langle \pi(X) \,u,v\rangle =-\langle u,\pi (X)\,v\rangle\,,\quad\forall u, v\in \mathcal{S}(\Bbb R^k),
\end{align*}
which implies the result.
\end{proof}
\end{prop}
\begin{remarque}\label{rem inverse}  
Proposition \ref{prop anti} is in particular true for any induced unitary representation  $ \pi_{\ell, \mathcal{H}} $  on $ \mathcal{G} $. As noted above, the induced representations
  indeed satisfy \eqref{def representation}.

\end{remarque}
\subsection{Characterization of hypoellipticity in the case of homogeneous invariant operators on stratified groups}
The purpose of this part is to provide the necessary and sufficient conditions for a polynomial operator of vector fields to be  a maximal hypoelliptic operator.
\begin{defi}
Let $ \mathcal{G} $ be a graded real Lie algebra which is stratified of type 2. We define the enveloping algebra $ \mathcal{U} (\mathcal{G}) $ as the noncommutative algebra of polynomial expressions of the following form:
\begin{align}\label{exp 1}
P=\sum_{\vert \alpha \vert \leq m} \, a_{\alpha} Y^{\alpha},
\end{align}
where $ a_{\alpha} \in \mathbb{C} $,  $ Y_{i, j} $ $ (i = 1, ..., p_ {j} \text{ and } j = 1,2) $ denotes a basis of $ \mathcal{G}_{j} $, $ \alpha = (\alpha_{1}, .., \alpha_{k}) $ is a $ k $-uplet of couples $ (i, j) $ with $ i \in \{1, ...., p_ {j} \} $, and 
$$Y^{\alpha}=Y_{\alpha_1}...Y_{\alpha_k} \, \text{ with } \, \vert \alpha \vert = \sum_{l=1}^{k}\, j(\alpha_{l}).  $$
When in equality \eqref{exp 1}, we consider only terms with $ \vert \alpha \vert = m $, the set of polynomial expressions of this form is denoted by $ \mathcal{U}_{m} ( \mathcal{G}) $.
\end{defi}
It is noted in \cite[Chapter 2]{helffer1980hypoellipticite} that the representation $ \pi_{\ell, \mathcal{H}} $ of the algebra $ \mathcal G$ naturally extends to a representation of the enveloping algebra $ \mathcal{U} (\mathcal{G}) $.
For all $ t> 0 $, we define an automorphism $ \delta_t $ of $ \mathcal{G} $ by the condition
$$\delta_t (a) = t^j \, a \quad \text{ if } a\in\mathcal{G}_j. $$
One can of course extend the definition of $ \delta_{t} $ (called family of dilations) to the enveloping algebra $ \mathcal{U} (\mathcal{G}) $ by setting
$$\delta_{t}(P)=\sum_{\vert \alpha \vert \leq m}a_{\alpha}\,(\delta_{t} Y )^{\alpha} =\sum_{\vert \alpha \vert \leq m}a_{\alpha}\,t^{\vert \alpha\vert } \, Y^{\alpha} \mbox{ for } t>0\,. $$
We note that, 
$$
\mathcal{P}\in \mathcal{U}_{m}(\mathcal{G})\text{ if and only if } \delta_{t} (\mathcal{P} )=\sum_{\vert \alpha \vert =m}\,a_{\alpha}\,(\delta_{t} Y)^{\alpha} =t^{m}\mathcal{P}, \, \forall t>0\,. 
$$
To any element $ Y $ of $ \mathcal {G} $, we can associate a left-invariant vector field $ \lambda (Y) $ on the group $ G $ defined by
$$
(\lambda (Y) \,f) (u) = \frac{d}{dt}\, f\left( u\cdot \exp (tY)\right)_{|_{t=0}}, \, \forall f \in H^1 (\mathbb{R}^{k}), \forall u \in G \,. $$
This correspondence makes it possible to identify the enveloping algebra $ \mathcal{U} (\mathcal{G}) $ with the algebra of all the polynomials of left-invariant vector fields. To $ P \in \mathcal{U} (\mathcal{G}) $, defined in \eqref{exp 1}, we can associate
$$ \lambda (P)=\sum_{\vert \alpha \vert \leq m}\, a_{\alpha} \lambda (Y)^{\alpha}.$$

We recall the theorem conjectured by C.~Rockland in \cite{rockland1978hypoellipticity}, proved by him in the case of the Heisenberg group, then in the general case by R.~Beals \cite{beals1976operateurs} for the necessary condition
  and by B.~ Helffer and J.~ Nourrigat in \cite{helffer1980hypoellipticite} for the sufficient condition. Note that the case of rank 3, which ultimately is the only one that will be useful here,  was previously obtained in \cite{helffer1978hypoellipticite}. Combined with a result of Rothschild-Stein \cite{rothschild1976hypoelliptic} in the particular case where the order  $ m $ of the operator $ \mathcal{P} $ is even and the Lie algebra is stratified of type $ 1 $ or $ 2 $, the Helffer-Nourrigat Theorem takes the following form.
\begin{thm} \label{thm type 2}
Let $ \mathcal{G} $ be a graded and stratified Lie algebra of type $ 1 $ or of type $ 2 $ and let $ \mathcal{P} \in \mathcal{U}_{m} (\mathcal{G}) $ with $ m $ even (just in the case of type $ 2 $). Then the following assertions are equivalent:
\begin{enumerate}
 \item The operator $ \mathcal{P} $  defined in \eqref{po} is hypoelliptic in $ G $.
\item The operator $ \mathcal{P} $ defined in \eqref{po} is maximal hypoelliptic in $ G $.
\item For any $ \pi $ non-trivial irreducible and unitary representation in $ \widehat{G} $,  the operator  $ \pi (\mathcal {P}) $ is injective in $ \mathcal{S}_{\pi} $, where $ \mathcal{S_{\pi}} $ denotes the space of the $ C^{\infty} $ vectors of the representation.
\end{enumerate}
\end{thm}
\begin{remarque} \label{rem 2.16}
When the H\"ormander condition \ref{ho} is satisfied, the condition (3) will be called the Rockland condition. To verify this condition, we observe that for any
  non-trivial irreducible and unitary  representation  $ \pi $ in $ \widehat{G} $ it suffices to show that if $ u $ satisfies $ \pi (\mathcal{P}) u = 0 $, then
\begin{itemize}
\item[$\bullet$] In the stratified case of Type 1, $$\pi (Y_{j}) \, u = 0 \,,\quad \forall j = 1, ..., p.$$ 
\item[$\bullet$] In the stratified case of Type 2,  \,$$\pi (X_{\ell}) = 0 \text{ and } \pi (Y_j) = 0,\, \forall \ell = 1 , .., p\, \text{ and }\,\forall
j = 1, .., q.$$
\end{itemize}
This implies indeed in the two cases that
$$ \pi (Y) u = 0, \quad \forall Y \in \mathcal{G}. $$
Then, assuming that $ \pi $ is irreducible and not trivial in Theorem \ref{thm type 2}, we get $ u = 0 $.
\end{remarque}
We will finish this part by quoting another result, from Helffer-Nourrigat in \cite{helffer1979approximation}, which appears in the proof of their theorem and which will be very useful to us.
\begin{thm} \label{thm 2.15}
  If $ \mathcal{P} \in \mathcal{U}_{m} (\mathcal{G}) $ is a maximal hypoelliptic operator, then there exists a strictly positive constant $ C $  such that, for any induced representation $ \pi = \pi_{\ell, \mathcal H} $, for all $ u $ in $ \mathcal S (\mathbb{R}^{k (\pi)}) $, we have the following maximal estimate:
\begin{align*}
\sum_{\vert \alpha \vert \leq m}\Vert \pi (Y^{\alpha}) u \Vert_{L^{2}(\mathbb{R}^{k})}^{2} \leq C\left( \Vert \pi(\mathcal{P}) u\Vert_{L^{2}(\mathbb{R}^{k})}^{2} + \Vert u\Vert_{L^{2}(\mathbb{R}^{k})}^{2} \right).
\end{align*}
\end{thm}

\subsection{Application to the maximal hypoellipticity of vector fields}
  We assume that the fields $ X_i $ and $ Y_{j} $ for $ i \in \{1, .., p \} $ and $ j \in \{1, .., q \} $ satisfy the H\"ormander condition \eqref{ho} in $ z_0 $. To the operator defined in \eqref{po}, we first associate in $ z_{0} $, an element of the enveloping algebra $ \mathcal{U} (\mathcal{G}) $ where $ \mathcal{G} $ denotes the free nilpotent Lie algebra with $ p + q $ generators $ (Z_1, ..., Z_ {p + q}) $ with rank $ r $ . Here, we follow Rothschild-Stein's approach in \cite{rothschild1976hypoelliptic}.
We associate to the operator $ P $  an element  $ \mathcal{P}_{z_{0}} $ of $ \mathcal{U}_m(\mathcal{G}) $ defined by
$$ \mathcal{P}_{z_0} = \sum_{\vert \alpha \vert = m} \, a_{\alpha} (z_0) \, Z^{\alpha} \,,$$
and we recall a result based on the articles \cite{rothschild1979criterion}, \cite {helffer1979caracterisation} and \cite {bolley1982condition}:
\begin{thm}[Theorem 0.7 in \cite{rothschild1979criterion}] \label{thm 1.11}  
Let $ \mathcal{P} $ be the operator defined in \eqref{po} satisfying the H\"ormander condition \ref{ho} at $ z_0 \in \Omega $. If $ \mathcal{P}_{z_0} $ satisfies the  Rockland's Criterion then the operator $ \mathcal{P}$  is maximal hypoelliptic in a neighborhood of $ z_0 $.
\end{thm}
The inverse is not true in general. It is the whole purpose of the book  \cite{helffer1980hypoellipticite} to give necessary and sufficient conditions for this maximal hypoellipticity.

\section{Proof of Theorem \ref{hypoelli0} when $ d = 2 $}
The proof consists of constructing $ \mathcal{G} $, a graded and stratified algebra of type $2$, and, at any point $ x \in \mathbb{T}^{2} $, an element $ \mathcal{K}_{x} $ of $ \mathcal{U} _ {2} (\mathcal{G}) $ which is hypoelliptic. The maximal estimate obtained for each $ \mathcal{K} _{x} $ will then be combined to give a maximal estimate for the operator $ K $.
~ \par To define $ \mathcal{K}_{x} $, we replace
  $ B_e (x) $ by a fixed constant $ \bf b \in \mathbb{R} $, and we will show a global estimate for the following model:
  \begin{align*}
K_\textbf{b} = v\cdot \nabla_{x} +\textbf{b} (v_1 \partial_{v_{2}} -v_2 \partial_{v_{1}}) -\Delta_v + v^{2}/4 -1, \text{ in } \mathbb{R}^{2}\times\mathbb{R}^{2},
\end{align*}
which appears as the image by an induced representation of $ \mathcal{K}_{x} $. 

  Then we will use a partition of unity and check the errors coming from the localization of the estimates.
\subsection{Application of Proposition \ref{prop 2.11}}
We will show that we can put ourselves within the framework of this proposition. 
We therefore look for a graded Lie algebra $ \mathcal G $ of type 2, a subalgebra $ \mathcal{H} $, an element $ \tilde{K}_\textbf{b} $ in $ \mathcal U_2 (\mathcal G) $ and a linear form $ \ell $ such that
 $$ 
  \pi_{\ell, \mathcal{H}} (\tilde K_\textbf{b})=K_{\textbf{b}}\,.
  $$
For this, one determines the necessary conditions on the brackets between the generating elements of this algebra.
  In the writing of $ K_{\textbf{b}} $ we can see the differential operators of degree $ 1 $ with the following polynomial coefficients:
  \begin{align}
&X_{1,1}^{'}=\partial_{v_1} && X_{1,1}^{''}=iv_{1}\\
&X_{2,1}^{'}=\partial_{v_2} && X_{2,1}^{''}=iv_2 \\
&X_{1,2}=v.\nabla_{x}\,.
\end{align}
$ K_{\textbf{b}} $ is indeed written as a polynomial of these five differential operators  
  \begin{align}\label{ecriture 17}
K_{\textbf{b}}=X_{1,2}&-\sum_{k=1}^{2}\, \left( (X_{k,1}^{'})^{2} +\frac{1}{4} (X_{k,1}^{''})^{2} -\frac{i}{4}(X_{k,1}^{'}X^{''}_{k,1} -X^{''}_{k,1}X_{k,1}^{'})    \right)\notag\\
&-i\,\textbf{b}\left( X_{1,1}^{'}X^{''}_{2,1}-X_{2,1}^{'}X^{''}_{1,1} \right).
\end{align}
We now look at the Lie algebra generated by these five operators and their brackets. This leads us to introduce three new elements that verify the following relations  :
\begin{align*}
& X_{2,2}:=[X_{1,1}^{'},X_{1,1}^{''}]=[X_{2,1}^{'},X_{2,1}^{''}]=i\,,\,\\
& X_{1,3}:=[X_{1,2},X_{1,1}^{'}]=\partial_{x_{1}} \,,\, X_{2,3}:=[X_{1,2},X_{2,1}^{'}]=\partial_{x_2}.
\end{align*}
We also observe that we have the following properties:
\begin{align*}
 &[X_{1,1}^{'},X_{2,1}^{'}]=[ X_{1,1}^{''},X_{2,1}^{''}]=0\,, && \\
 &[ X_{j,1}^{'},X_{k,3}]=[ X_{j,1}^{''},X_{k,3}]=[ X_{k,3},X_{2,2}]=...=0&\,,\, \forall j,k=1,2\,.&
 \end{align*}
 
 We then construct a graded Lie  algebra $ \mathcal{G} $ verifying the same commutator relations. More precisely, $ \mathcal{G} $ is stratified of type $ 2 $, nilpotent of rank $ 3 $, its underlying vector space is $ \mathbb{R}^{8} $, and $ \mathcal{G}_{1} $ is generated by $ Y_{1,1}^{'}, Y^{'}_{2,1}, Y_{1,1}^{''} $ and $ Y^{'' }_{2,1} $, $ \mathcal{G}_{2} $ is generated by $ Y_{1,2} $ and $ Y_{2,2} $ and $ \mathcal{G}_{3} $ is generated by $ Y_{1,3} $ and $ Y_{2,3} $. The laws of algebra are given by
\begin{align}
& Y_{2,2}=[Y_{1,1}^{'},Y_{1,1}^{''}]=[Y_{2,1}^{'},Y_{2,1}^{''}]\,,\
 Y_{1,3}=[Y_{1,2},Y_{1,1}^{'}]\,, \, Y_{2,3}=[Y_{1,2},Y_{2,1}^{'}]\,,\label{annu 0}\\
 &[Y_{1,1}^{'},Y_{2,1}^{'}]=[ Y_{1,1}^{''},Y_{2,1}^{''}]=0\,, \label{annu 1} \\
 &[ Y_{j,1}^{'},Y_{k,3}]=[ Y_{j,1}^{''},Y_{k,3}]=[ Y_{k,3},Y_{2,2}]=...=0\,\, \forall j,k=1,2\,.\label{annu 2}
 \end{align}
 We check that the mapping $ \pi $ (with the convention that if $ \diamond = \emptyset $ there is no exponent) defined on its basis by 
 \begin{align}
\pi(Y_{i,j}^{\diamond})= X_{i,j}^{\diamond}\text{ with }i=1,2, j=1,2,3\text{ and }\diamond\in \{ \emptyset, \, \prime, \prime\prime\}.
\end{align}
defines a representation of the Lie algebra $ \mathcal{G} $. 

 We now see that our representation $ \pi $ can be rewritten in the following form
 \begin{align*}
\pi(Y)&=P_{1}(Y)\partial_{v_1}+P_{2}(v_1; Y)\partial_{v_2}+P_{3}(v_1, v_2;Y)\partial_{x_1}+P_{4}(v_1,v_2,x_1;Y)\partial_{x_2}\\
&\qquad +iQ(v_1,v_2,x_1,x_2; Y)\,,
\end{align*}
where for all $ Y \in \mathcal{G} $, there are $ a, b, c, d, \alpha, \beta, \gamma, $ and $ \delta \in \mathbb{R} $ such that
\begin{align*}
Y&=aY^{'}_{1,1}+bY^{'}_{2,1}+cY_{1,1}^{''}+dY^{''}_{2,1}\\
 &\qquad +\alpha Y_{1,2}+\beta \,Y_{2,2}\\
 &\qquad + \gamma Y_{1,3}+\delta Y_{2,3},
\end{align*} 
 and the polynomials $ P_{j} $ with $ j = 1, .., 4 $ and $ Q $ are defined by
\begin{align*}
&P_{1}(Y)=a, \,P_{2}(v_1;Y)=b, \, P_{3}(v_1,v_2;Y)=\alpha v_1+\gamma\\
&P_{4}(v_1,v_2,x_1;Y)=\alpha v_2+\delta  \text{ and } Q(v,x;Y)=c v_1 +d v_2 +\beta.
\end{align*}
We can now apply Proposition \ref{prop 2.11}. We then obtain $ \pi = \pi_{\ell, \mathcal{H}} $ with
  \begin{align*}
&\ell \in \mathcal{G}^{*}, \, \langle \ell, Y \rangle = Q(0;Y)=\beta,\\
&\mathcal{H}:=\{Y\in \mathcal{G}/\, P_1 (0;Y)=...=P_4 (0;Y)=0\}
  = \mathrm{Vect} (Y_{1,1}^{''}, Y_{2,1}^{''}, Y_{1,2}, Y_{2,2}), 
\end{align*}

Let's go back to our operator $ K_{\textbf{b}} $ which was written as a polynomial of the vector fields $\{X_{j,2},X_{j,3},X_{j,1}^{'},X_{j,1}^{''}\}_{j=1,2}$ in \eqref{ecriture 17}. We then define $ \tilde K_\textbf{b} $ as the same polynomial but in this time the operator  is function of the vector fields  $\{Y_{j,2},Y_{j,3},Y_{j,1}^{'},Y_{j,1}^{''}\}_{j=1,2}$ defined as follows
\begin{align}
\tilde{K}_{\textbf{b}}=Y_{1,2}&-\sum_{k=1}^{2}\, \left( (Y_{k,1}^{'})^{2} +\frac{1}{4} (Y_{k,1}^{''})^{2} -\frac{i}{4}(Y_{k,1}^{'}Y^{''}_{k,1} -Y^{''}_{k,1}Y_{k,1}^{'})    \right)\\
&-i\,\textbf{b}\left( Y_{1,1}^{'}Y^{''}_{2,1}-Y_{2,1}^{'}Y^{''}_{1,1} \right)\notag 
\end{align}
 such that
$$
\pi_{\ell, \mathcal{H}}(\tilde K_\textbf{b}) =K_\textbf{b}\,.
$$
Note that with the notation used in the introduction to the section, we have:
$$ 
\mathcal{K}_{x} = \widetilde K_{B_e(x)}\,.
$$
~
\subsection{Verification of Rockland's Criterion}~\\
To prove the maximal hypoellipticity, we must check the Rockland criterion (see Theorem \ref{thm type 2}) for $ \tilde K_\textbf{b} $.
  Let $ \pi $ be a unitary irreducible non-trivial representation of $ G $ in $ V_\pi $. We will show that the operator $ \pi (\tilde{K}_{\textbf{b}}) $ is an injective operator in the space $ \mathcal{S}_{\pi} $, which identifies when $ \pi $ is irreducible to $ \mathcal{S} (\mathbb{R}^{k (\pi)}) $. Let $ u \in \mathcal{S} (\mathbb{R}^{k (\pi)}) $ such that
$$\pi(\tilde{K}_{\textbf{b}}) u=0. $$ 
On the one hand, we have
$$\mathrm{Re} \langle \pi (\tilde{K}_{\textbf{b}})u,u\rangle =0. $$
On the other hand, by integration by parts and by using that the operator $ \pi (Y) $ is a formally skew-adjoint operator (see Proposition \ref{prop anti}), we obtain
\begin{align*}
\mathrm{Re} \langle \pi(\tilde{K}_{\textbf{b}})u,u\rangle &= \underbrace{\mathrm{Re} \langle \pi (Y_{1,2})u,u \rangle}_{\mathrm{I}}+\sum_{k=1}^{2}\,\left( \Vert \pi(Y^{'}_{k,1})u\Vert^{2} +\Vert \pi(Y_{k,1}^{''})u\Vert^{2} \right) \\
&-\textbf{b}\,\underbrace{\mathrm{Re} \left\langle i \left( \pi(Y_{1,1}^{'})\pi(Y^{''}_{2,1})-\pi(Y_{2,1}^{'})\pi(Y^{''}_{1,1}) \right)u, u\right\rangle}_{\mathrm{II}}.
\end{align*}
First, we will calculate the term $ \mathrm{I}  $. Using the fact that the operator $ \pi (Y_{1,2}) $ is skew-adjoint according to Proposition \ref{prop anti}, we obtain
$$  \langle \pi (Y_{1,2})u,u \rangle =- \langle u, \pi(Y_{1,2})u\rangle=-\overline{\langle \pi (Y_{1,2})u,u \rangle}\,. $$
Therefore, we have
$$\langle \pi (Y_{1,2})u,u \rangle + \overline{\langle \pi (Y_{1,2})u,u \rangle} =0, $$
so 
$$\mathrm{I}=\mathrm{Re} \langle \pi (Y_{1,2})u,u \rangle =\left(\langle \pi (Y_{1,2})u,u \rangle + \overline{\langle \pi (Y_{1,2})u,u \rangle}\right)/2 =0.$$
Then we go to calculating the term $ \mathrm{II} $. Using that $ \pi $ is a representation and the relations of the given commutators in \eqref{annu 1}, we get
 \begin{align*}
&[\pi(\tilde{Y}_{1,1}^{'}),\pi(\tilde{Y}^{''}_{2,1})]= \pi \left([\tilde{Y}_{1,1}^{'},\tilde{Y}^{''}_{2,1} ] \right)=0\\
&[\pi(\tilde{Y}_{2,1}^{'}),\pi(\tilde{Y}^{''}_{1,1}) ]=\pi\left([\tilde{Y}_{2,1}^{'}, \tilde{Y}^{''}_{1,1}] \right) =0\,.
\end{align*}
Then the operators $  i  \pi(Y_{1,1}^{'})\pi(Y^{''}_{2,1})$ and $i \pi(Y_{2,1}^{'})\pi(Y^{''}_{1,1})$  are skew-adjoint on the space $ \mathcal{S}_{\pi} $ with the scalar product of $ V_\pi $ for any representation $ \pi $. By integration by parts, we have
\begin{align*}
&\left\langle i  \pi(Y_{1,1}^{'})\pi(Y^{''}_{2,1}) u, u\right\rangle=-\overline{\left\langle i  \pi(Y_{1,1}^{'})\pi(Y^{''}_{2,1}) u, u\right\rangle },\\
&\left\langle i  \pi(Y_{2,1}^{'})\pi(Y^{''}_{1,1}) u, u\right\rangle=-\overline{\left\langle i  \pi(Y_{2,1}^{'})\pi(Y^{''}_{1,1}) u, u\right\rangle },
\end{align*}
and then, 
\begin{align*}
\mathrm{Re}\left\langle i \pi(Y_{1,1}^{'})\pi(Y^{''}_{2,1}) u,u\right\rangle=\mathrm{Re} \left\langle i \pi(Y_{2,1}^{'})\pi(Y^{''}_{1,1}) u,u\right\rangle =0\,.
\end{align*}
 Therefore, we have
$$\mathrm{II}=\mathrm{Re} \left\langle i\left( \pi(Y_{1,1}^{'})\pi(Y^{''}_{2,1}) - \pi(Y_{2,1}^{'})\pi(Y^{''}_{1,1}) \right)u,u\right\rangle =0\,. $$
The identity
$$ \mathrm{Re} \langle \pi(\tilde{K}_{\textbf{b}})u,u\rangle = \sum_{k=1}^{2}\,\left( \Vert \pi(Y^{'}_{k,1})u\Vert^{2} +\Vert \pi(Y_{k,1}^{''})u\Vert^{2} \right),$$
 implies 
\begin{align}
\pi(Y_{j,1}^{'})u=\pi(Y_{j,1}^{''})u=0\,,\quad \forall j=1,2.\label{ég 9}
\end{align}

It remains to consider $ Y_{1,2} $, for which we can notice that
\begin{align*}
Y_{1,2}&=\tilde{K}_{\textbf{b}} +\sum_{k=1}^{2}\, \left( (Y_{k,1}^{'})^{2} +\frac{1}{4} (Y_{k,1}^{''})^{2} -\frac{i}{4}[Y_{k,1}^{'},Z^{''}_{k,1}] \right)\\
&\qquad +i\textbf{b}\left( Y_{1,1}^{'}Y^{''}_{2,1}-Y_{2,1}^{'}Y^{''}_{1,1} \right) .
\end{align*}
Applying $ \pi $ and by action on $ u $, we have
\begin{align*}
\pi(Y_{1,2})u&= \pi(\tilde{K}_{\textbf{b}})u +\sum_{k=1}^{2}\, \pi\left( (Y_{k,1}^{'})^{2} +\frac{1}{4} (Y_{k,1}^{''})^{2} -\frac{i}{4}[Y_{k,1}^{'},Y^{''}_{k,1}] \right)u\\
&\qquad +i\textbf{b}\,\pi\left( Y_{1,1}^{'}Y^{''}_{2,1}-Y_{2,1}^{'}Y^{''}_{1,1} \right)u\\
&=0\, .
\end{align*}
From Remark \ref{rem 2.16} in the stratified case of type 2,  we deduce that
$$ \pi (Y)u=0 \,,\, \forall Y\in \mathcal{G}\,,$$
which implies, $\pi$ being assumed to be non trivial,  $u=0$.\\

The operator $ \pi (\tilde{K}_{\textbf{b}}) $  is therefore injective in the $ \mathcal{S}_{\pi} $ for any irreducible and non-trivial representation $ \pi $. Therefore, according to Theorem \ref{thm type 2} the operator $ \tilde{K}_{\textbf{b}} $ is maximal hypoelliptic in the group $ G $.
By applying Theorem \ref{thm 2.15} with $ K_\textbf{b} = \pi_{\ell, \mathcal{H}} (\tilde{K}_\textbf{b} $), we obtain the existence of $ C> 0 $ such that
\begin{align}\label{ing 11}
\Vert X_{1,2} u\Vert +&\sum_{k=1}^{2} \, \left(\Vert (X_{k,1}^{'})^{2}u\Vert  +\Vert (X_{k,1}^{''})^{2}u\Vert \right) +\sum_{k,\ell=1}^{2} \Vert X^{'}_{k,1}X^{''}_{\ell, 1} u \Vert\notag \\
&  \leq C\, \left( \Vert K_{\textbf{b}}u\Vert+\Vert u \Vert \right)\,,
\end{align}
the previous inequality is verified.
\subsection{Proof of Theorem \ref{hypoelli0}}
\subsubsection{First step.}
We fix $ x_0 \in \mathbb{R}^{2} $ and we write $ B_e (x_0) = \textbf{b}_0 \in \mathbb{R} $. We recall that
$$K=v\cdot \nabla_x -(v\wedge \mathcal{B}_e )\cdot \nabla_v -\Delta_v + v^2/4 -1. $$
We begin with  $u \in \mathcal{S}(\mathbb{R}^{2}\times \mathbb{R}^{2})$ with $ \mathrm{P}_x(\mathrm{Supp}\, u)\subset B(x_0, \varepsilon)$,
where $\mathrm{P}_x$ is the standard projection with respect to $x$ variables
and $\varepsilon>0$ a constant which will be chosen later, we have
$$K u=(K-K_{\textbf{b}_{0}} ) u +K_{\textbf{b}_{0}} u. $$
On the one hand, we have by definition of $K$ and $K_{\textbf{b}_0}$
\begin{align}\label{ing 12}
\Vert (K-K_{\textbf{b}_{0}} )u\Vert &=\Vert (\mathcal{B}_e -\mathcal{B}_e (x_0))(v_1\partial_{v_2} - v_2 \partial_{v_1} ) u \Vert\notag \\
&\leq \Vert \nabla_{x}B_e \Vert_{L^{\infty}}\, \vert x-x_0\vert \, \Vert (v_1\partial_{v_2} - v_2 \partial_{v_1} )u\Vert\\
&\leq  \varepsilon\, \,\Vert \nabla_{x}B_e \Vert_{L^{\infty}}\,  \Vert (v_1\partial_{v_2} - v_2 \partial_{v_1} )u \Vert\notag\,.
\end{align}
We then obtain the following estimate:
$$\Vert (K-K_{\textbf{b}_{0}} )u\Vert \leq \varepsilon\, \Vert B_e \Vert_{\mathrm{Lipsch}(\mathbb{T}^{2})}\,  \Vert u\Vert_{\tilde{B}^2} \,.$$
On the other hand, by the inequality \eqref{ing 11} and Theorem \ref{thm 1.11} at the point $ x_0 $, we obtain that the operator $ K_{\textbf{b}_0} $ is maximal hypoelliptic in $ B (x_0 , \varepsilon) $. So the operator $ K_{\textbf{b}_0} $ verifies, for a constant $ C> 0 $ large enough,  which depends on the coefficients of the polynomial operator of vector fields $K_{\textbf{b}_0}$, the following estimate:
 \begin{align}\label{ing 13}
\Vert K_{\textbf{b}_{0}}u\Vert+\Vert u \Vert&\geq \frac{1}{C}\Vert X_{1,2} u\Vert +\frac{1}{C}\sum_{k=1}^{2} \, \left(\Vert (X_{k,1}^{'})^{2}u\Vert +\Vert (X_{k,1}^{''})^{2}u\Vert \right)+\frac{1}{C}\sum_{k,\ell=1}^{2} \Vert X^{'}_{k,1}X^{''}_{\ell, 1} u \Vert,
\end{align}
Finally, we observe that
\begin{align*}
\Vert K u\Vert \geq \Vert K_{\textbf{b}_{0}} u \Vert - \Vert (K-K_{\textbf{b}_{0}} )u  \Vert\,,
\end{align*}
Using the inequalities \eqref{ing 12} and \eqref{ing 13} and choosing $ 0 <\varepsilon <\displaystyle\frac{1}{C\,\Vert B_e \Vert_{\mathrm{Lipsch}(\mathbb{T}^{2})}} $, we get $ \tilde C> 0 $ such that for all $u \in \mathcal{S}(\mathbb{R}^{2}\times \mathbb{R}^{2})$ with $\mathrm{P}_x(\mathrm{Supp}\, u)\subset B(x_0, \varepsilon)$
\begin{align}\label{ing 14}
\Vert X_{1,2} u\Vert +\sum_{k=1}^{2} \, &\left(\Vert (X_{k,1}^{'})^{2}u\Vert +\Vert (X_{k,1}^{''})^{2}u\Vert \right)+\sum_{k,\ell=1}^{2} \Vert X^{'}_{k,1}X^{''}_{\ell, 1} u \Vert \notag\\&
\qquad \leq \tilde C\, \left( \Vert K u\Vert+\Vert u \Vert \right)\,.
\end{align}
\subsubsection{Second step.}
We have shown that for each $x_0 \in \Bbb T^{2}$ there exists a $B(x_0, \varepsilon(x_0))$ such that \eqref{ing 14} applies. $\Bbb T^{2}$ being compact, then there exists a smooth partition of unity $(\varphi_j)_{j=1}^{n_0}$ in the variable $x\in \Bbb T^2$ such that 
$$\varphi_j \in C_0^\infty (\Bbb T^{2})\, \text{ and } \, \sum_{j=0}^{n_0} \varphi_j^2 = 1 \, \text{ on }\, \Bbb T^{2},$$
 and the inequality \eqref{ing 14} holds on each $\mathrm{Supp} \varphi_j$ (with a uniform constant since the collection of $\varphi_j$ is finite). More precisely,  we take $\varepsilon =\inf \{\varepsilon_j, \, j=1,..,n_0\}>0,$ where 
$\varepsilon_j $ is the radius of the ball centered at a point $ x_j $ whose \eqref{ing 14} holds,  such that 
 $$ \mathrm{Supp} \, \varphi_j \subset B(x_j ,\varepsilon) \subset B(x_j ,\varepsilon_j) \quad \forall j\in \{1,..,n_0\},$$
where the ball $B(x_j,\varepsilon)$ is defined by for all $ j=1,..,n_0$
$$B(x_j, \varepsilon):=\{x\in \mathbb{T}^{2}/\, d(x,x_j)<\varepsilon\}. $$ 
\textbf{Intermediate step.} 
The purpose of this step is to show the following two estimates:
\begin{align}
&\forall \eta>0, \exists C_{ \eta}>0\, \text{ such that } \, \forall u\in \tilde{B}^2 ,\Vert u\Vert_{\tilde{B}^{1}}^2\leq \eta \Vert u\Vert_{\tilde{B}^2}^2 +C_{\eta} \Vert u\Vert^2\,,  \label{est t1} \\
&\forall \varphi\in C^{\infty} (\mathbb{T}^2_{x}), \exists C_{\varphi}>0 \,\text{ such that }\, \Vert [K,\varphi]u\Vert \leq C_{\varphi}\Vert u\Vert_{\tilde{B}^{1}}\,, \forall u\in \tilde{B}^1,\label{est t2}
\end{align}
where $\tilde{B}^1$ is the space $  L_{x}^{2} \widehat \otimes B^{1}_{v} $ with the following Hilbertian norm:
$$ \Vert u\Vert^2_{\tilde{B}^1}:=\sum\limits_{\vert \alpha \vert +\vert \beta \vert \leq 1}\Vert v^\alpha \partial_v^\beta \,u\Vert^2, \, \forall u\in \tilde{B}^1.$$
According to the definition of $ \Vert \cdot \Vert_{B^1_v} $, we have for $u\in \mathcal S(\mathbb T^2\times \mathbb R^2)$ and $x\in \Bbb T^2$
\begin{align*}
\Vert u (x,.)\Vert_{B^1_v}^{2}&=\sum_{\vert \alpha\vert +\vert \beta  \vert \leq 1}\, \Vert v^{\alpha}\partial_{v}^{\beta} u(x,.)\Vert_{L^{2}(\mathbb{R}^2)}^{2} \\
&=\sum_{\vert \alpha\vert +\vert \beta  \vert \leq 1} \langle v^{\alpha}\partial_{v}^{\beta} u(x,.), v^{\alpha}\partial_{v}^{\beta} u(x,.)\rangle\,.
\end{align*}
 By an integration by parts with respect to $v$ and by applying the Cauchy-Schwarz inequality, then there exists a constant $C>0$ such that 
\begin{align*}
\Vert u (x,.)\Vert_{B^1_v}^{2}&\leq C\Vert u(x,.)\Vert \, \Vert u(x,.)\Vert_{B^{2}_v}\,.
\end{align*}
Finally, by the Young inequality we obtain that for all $ \eta> 0 $, there exists $ C_{\eta}> 0 $ such that 
$$ \forall \eta >0, \, \Vert u(x,.)\Vert_{B^{1}_v}^{2} \leq \eta \Vert u(x,.)\Vert_{B^2_v}^{2} +C_{\eta} \Vert u(x,.)\Vert^{2}\,.$$
 Then, by integrating in $ x $ on the torus $ \mathbb{T}^2 $, we deduce the estimate \eqref{est t1}.
~ \par Now, we note that the commutator is
$$[K, \varphi]=[v\cdot\nabla_{x}, \varphi]=v\cdot\nabla_x \varphi\,,$$
according to the previous equality and applying H\"older's inequality ($p=1$ and $q=\infty$), we obtain
\begin{align*}
\Vert [K,\varphi]u\Vert = \Vert v\cdot \nabla_x \varphi u\Vert & \leq \Vert \nabla_x \varphi\Vert_{L^{\infty} (\mathbb{T}^2)}\,\Vert v\, u\Vert_{L^{2}(\mathbb{T}^{2}\times\mathbb{R}^{2})} \\
&\leq \Vert \nabla_x \varphi\Vert_{L^{\infty} (\mathbb{T}^2)}\,\Vert u\Vert_{\tilde{B}^{1}}\\&
\leq C_{\varphi} \Vert u\Vert_{\tilde{B}^1}\,. 
\end{align*}

\textbf{End of the proof.}
Let $ B_e \in \mathrm{Lipsch} (\mathbb{T}^{2}) $ and $  u \in \mathcal S (\mathbb{T}^{2} \times \mathbb{R}^{2}) $. By the partition of unity and by application of the inequality \eqref{ing 14} with $w=u\varphi_j$, we obtain that there exists $ C> 0 $ such that
\begin{align}\label{est t3}
 \Vert u \Vert_{\tilde{B}^{2}}^{2} =\sum_{j=1}^{n_0} \, \Vert \varphi_j u \Vert^{2}_{\tilde{B}^2} \leq C\sum_{j=1}^{n_0}\,\left( \Vert K\varphi_j u \Vert^{2}+\Vert \varphi_j u \Vert^{2} \right), 
\end{align}
Using the definition of the commutators, we get
\begin{align*}
\sum_{j=1}^{n_0}\,\Vert K\varphi_j u \Vert^{2}& \leq 2 \sum_{j=1}^{n_0}\left(\Vert [\varphi_j , K]u\Vert^{2} +\Vert \varphi_j Ku \Vert^2\right)\,.
\end{align*}
According to the estimates \eqref{est t1} and \eqref{est t2}, we obtain, for every $ \eta> 0 $, the existence of $ C_\eta> 0 $ such that 
\begin{align}\label{ing 15}
\sum_{j=1}^{n_0}\,\Vert K\varphi_j u \Vert^{2}\leq 2 \Vert K u\Vert^{2} +2\,\eta \Vert u\Vert_{\tilde{B}^2}^2 +C_{\eta} \Vert u\Vert^2\,.
\end{align}
By inserting inequality \eqref{ing 15} into inequality \eqref{est t3}, we deduce that, for every $ \eta> 0 $, there exists $ C_\eta> 0 $ such that
\begin{align*}
\Vert u \Vert_{\tilde{B}^{2}}^{2}\leq 2 \,C\Vert K u\Vert^{2} +2\,\eta \,C\Vert u\Vert_{\tilde{B}^2}^2 +C_{\eta} \Vert u\Vert^2,
\end{align*}
Then, using the same techniques, we have
\begin{align*}
\Vert v\cdot \nabla_{x} u \Vert^{2}&=\sum_{j=1}^{n_0}\, \Vert \varphi_j (v\cdot \nabla_{x}u) \Vert^{2} \\
&  \leq 2 \sum_{j=1}^{n_0}\, \left( \Vert [\varphi_j, v\cdot\nabla_{x}]u\Vert^{2} + \Vert (v\cdot \nabla_{x})(\varphi_j u )\Vert^{2}\right)
\end{align*}

Using the inequalities \eqref{ing 15} and \eqref{ing 14}, we get $ C'> 0 $ and for all $ \eta'> 0 $, $ C_{\eta'}>0$ such that
\begin{align*}
\Vert v\cdot\nabla_{x} u \Vert^{2} &\leq \sum_{j=1}^{n_0} \left( C'\Vert \varphi_j \, Ku \Vert^{2} +\eta'\Vert \varphi_j u \Vert_{\tilde{B}^2}^{2} +C_{\eta'}\Vert\varphi_j u \Vert^{2}\right)\\
&\leq C' \Vert Ku\Vert^{2} +\eta'\Vert u \Vert^{2}_{\tilde{B}^2}+ C_{\eta'}\Vert u \Vert^{2}.
\end{align*}
It remains to estimate the following term:
\begin{align*}
\Vert B_e ( v_1 \partial_{v_2} - v_2 \partial_{v_1}) u\Vert^{2} =\sum_{j=1}^{n_0}\,\Vert B_e ( v_1 \partial_{v_2} - v_2 \partial_{v_1})(\varphi_j u )\Vert^{2}.
\end{align*}
By direct application of the inequalities \eqref{ing 14} and \eqref{ing 15}, we obtain that there exists $ C, C''> 0 $ and for all $ \eta''> 0 $, $ C_{\eta''}> 0 $ such that 
 \begin{align*}
\Vert B_e ( v_1 \partial_{v_2} - v_2 \partial_{v_1})u\Vert^{2}&\leq C\sum_{j=1}^{n_0} \left(\Vert K(\varphi_j u )\Vert^{2}+\Vert u \Vert^{2} \right)\\
&\leq C'' \Vert Ku  \Vert^{2}+\eta''\Vert u \Vert^{2}_{\tilde{B}^2} + C_{\eta''}\Vert u \Vert^{2}.
\end{align*}
Choosing $ \eta $, $ \eta' $ and $ \eta''> 0 $ such that
$ 2C\,\eta + \eta' + \eta'' <1 $, we thus obtain the existence of a constant $ \tilde{C}> 0 $ such that 
\begin{align*}
\Vert (v\cdot \nabla_{v} -B_e ( v_1 \partial_{v_2} - v_2 \partial_{v_1}))u\Vert + \Vert u \Vert_{\tilde{B}^2} \leq \tilde{C}(\Vert K u \Vert +\Vert u \Vert ), \quad \forall u \in  \mathcal S(\mathbb{T}^{2}\times \mathbb{R}^{2}).
\end{align*}
\section{Proof of Theorem \ref{hypoelli0} when $d=3$}
\subsection{Preliminaries}
We now present the proof of Theorem \ref{hypoelli0} when $ d = 3 $, and  explain the differences with the proof in the case $d=2$. We first replace $ B_e (x) $ with a constant vector $ \textbf{b}= (b_1, b_2, b_3) \in \mathbb {R}^{3} $ fixed, and we show an overall estimate for the following model:
\begin{align}\label{f-p-fixe}
K_{\textbf{b}} = v\cdot \nabla_{x} +(v\wedge \textbf{b} )\cdot \nabla_{v} -\Delta_v + v^{2}/4 -3/2.
\end{align}
We define the transformation
$$V_{M}:\mathbb{R}^{3}\times\mathbb{R}^{3}\to \mathbb{R}^{3}\times \mathbb{R}^{3}, \quad 
(x,v)\mapsto V_{M}(x,v):=(Mx,( M^{-1})^{t}v),$$
where $ M $ is the rotation matrix that is obtained by multiplying the two matrices of rotations $ R ({\theta_1}) $ and $ R (\theta_2) $ around the axes $ (Ox) $ and $ (Oz) $ with rotation angles $ \theta_1 $ and $ \theta_2 $ respectively
\begin{align*}
R(\theta_1):=\begin{pmatrix}
1 &0&0\\
0&\cos \theta_1&-\sin \theta_1\\
0&\sin \theta_1 &\cos \theta_1
\end{pmatrix} \quad \quad 
R(\theta_2):=\begin{pmatrix}
\cos \theta_2 &-\sin \theta_2 &0\\
\sin \theta_2 &\cos\theta_2 &0\\
0&0&1
\end{pmatrix}
\end{align*}
where both angles are defined by
$$\theta_1 =\arctan (b_{1}\,/\,b_{2})\quad,\quad \theta_2 =\arctan (\sqrt{b_{1}^{2}+b_{2}^{2}}\,/\,b_3). $$
We note that the following differential operators are invariant by conjugation by the transformation $ V_{M} $ (by orthogonality of matrix $M$ that is to say $M^{-1}=M^{t}$):
\begin{align*}
V_{M} \,(v\cdot \nabla_{x})\, V_{M^{-1}} =(M^{-1})^{t}\,v\cdot (M^{-1})^{t} \nabla_{x} =(M^{-1})^{t}\,M^{-1}\,v\cdot \nabla_{x}=v\cdot\nabla_{x}\,
\end{align*}
\begin{align*}
V_{M} \, (-\Delta_{v} + v^{2}/4 - 3/2)\,V_{M^{-1}}&=V_{M}\,(-\nabla_{v}+v/2)\cdot(\nabla_{v}+v/2)\,V_{M^{-1}}\\
&=(M^{-1})^{t}\,(-\nabla_{v}+v/2)\cdot M^{t} \,(-\nabla_{v}+v/2)\\
&=M\,M^{t}\,\left(-\Delta_{v}+v^{2}/4 -3/2\right)\\
&=-\Delta_{v}+v^{2}/4 -3/2.
\end{align*}
By construction of the matrix $ M $, where $\vert \textbf{b} \vert$ is the Euclidiean norm of $\textbf{b}\in \mathbb{R}^3$,  conjugation of the magnetic operator gives
$$V_{M}(v\wedge \textbf{b})\cdot\nabla_{v} \,V_{M^{-1}}= \vert \textbf{b}\vert (v_1 \partial_{v_2} -v_2 \partial_{v_1}). $$
Therefore, the conjugation of the Fokker-Planck-magnetic operator defined in \eqref{f-p-fixe}, by the canonical transformation operator $ V_{M} $, gives us the following operator:
\begin{align}\label{op trans}
Q_{\vert \textbf{b}\vert}:= V_{M}\,K_{\textbf{b}}\,V_{M^{-1}}=v\cdot \nabla_x + \vert \textbf{b}\vert \,(v_{1}\partial_{v_2} -v_2 \partial_{v_1 })-\Delta_v + v^{2}/4 -3/2,
\end{align}
~\par Note that the space $ \tilde{B}^{2} $ is invariant by rotation. The maximal estimates for $ K_\textbf{b} $ are therefore equivalent to the maximal estimates for $ Q_{\vert \textbf{b} \vert} $ (with uniform control of constants). 

As in the case $ d = 2 $, the proof consists in constructing a graded Lie algebra $ \mathcal{G} $ of rank $ 3 $, and, in every point $ x \in \mathbb{T}^{3} $, an element $ \mathcal{Q}_{\vert \textbf{b} \vert} $ of $ \mathcal{U}_{2} (\mathcal{G}) $ hypoelliptic with $\textbf{b}=B_e(x)$. We deduce from the maximal estimate obtained for each $ \mathcal{Q}_{\vert \textbf{b} \vert} $, a maximal estimate for the operator $ Q_{\vert \textbf{b} \vert} $.

By conjugation by the canonical transformation $ V_{M^{- 1}} $ of the operator $ Q_{\vert \textbf{b} \vert} $, we then have
\begin{align}\label{ing 38}
&\Vert X_{1,2} u\Vert +\sum_{k=1}^{3} \, \left(\Vert (X_{k,1}^{'})^{2}u\Vert +\Vert (X_{k,1}^{''})^{2}u\Vert\right)+\sum_{k,\ell =1}^{3}\, \Vert X_{k,1}^{'}X_{\ell,1}^{''}u\Vert\notag \\
&\qquad \qquad \leq C\left( \Vert K_{\textbf{b}}u\Vert+\Vert u \Vert \right),\quad \forall u \in S(\mathbb{R}^{3}\times \mathbb{R}^{3}).
\end{align} 
\subsection{End of the proof}
For the rest of the proof, there is no difference with the case ~ $ d = 2 $, we use a partition of unity and we control the error, in order to obtain the maximal estimate \eqref{hypomax} for the initial operator $ K $ in the space $ \mathcal{S} (\mathbb{T}^{3} \times \mathbb{R}^{3}) $. For more details on a similar result without a magnetic field but with an electric potential, we refer the reader to chapter 9 of the book of B. ~Helffer and F.~Nier \cite{nier2005hypoelliptic}.
\appendix
\section{Maximal accretivity of magnetic Fokker-Planck operator with low regularity} 
\subsection{Preliminary remark}
By following the steps of the proof of the maximal accretivity for Kramers-Fokker-Planck operator without a magnetic field, given in \cite[Proposition 5.5]{nier2005hypoelliptic}, we need some local hypoelliptic regularity of the following operator:
$$
  v\cdot \nabla_x -(v\wedge B_e )\cdot \nabla_v -\Delta_v =Y_0+\sum_{j=1}^{d}Y_j^2\,,
  $$
  where 
   $$ Y_0= v\cdot \nabla_x -(v\wedge B_e )\cdot \nabla_v \, \text{ and }\,  Y_j=\partial_{v_j}\,,\forall j=1,..,d\,. $$
In the case where $ B_e $ is $ C^\infty $, this regularity results immediately from the hypoellipticity argument of H\"ormander operators.
   The difficulty in our weakly regular case is that the vector field $ Y_0 $ has coefficients in $ L^\infty $. We can only hope for a weaker Sobolev type of regularity.
\subsection{Review of Sobolev regularity for H\"ormander operators of type-2} 
In this part, we will recall a result of Sobolev regularity for a relevant class of differential operators. 
We consider the differential operator of type-2 of H\"ormander given by
\begin{align*}
\mathcal{L}=\sum_{j=1}^{n}\,X_j^2 +X_0\,,
\end{align*}
where the vector fields $ X_0, .., X_n $ are real and $ C^\infty $ on an open set $ \Omega \subset \mathbb{R}^{n} $. We suppose further that the $ X_j $ with $ j = 0,1, .., n $ satisfy H\"ormander condition~\ref{ho}.
 The hypoellipticity of these operators has been studied by L.~H\"ormander in \cite{hormander1967hypoelliptic}. 

We denote by $ \mathcal{H}^m_{loc} (\mathbb{R}^{2d}) $ the space of functions locally in $ \mathcal{H}^m (\Bbb R^{2d})$ defined in \eqref{espace_bis}\,.
 (See \cite{rothschild1976hypoelliptic} for more details of this subject).  We recall from Rothschild-Stein in \cite{rothschild1976hypoelliptic} the following result.

\begin{thm}[Theorem 18 in \cite{rothschild1976hypoelliptic}]\label{thm_transitoire}
If $f, \mathcal{L}\,f\in L^{2}_{loc}(\Omega)$ then $ f\in \mathcal{H}^2_{loc}(\Omega)$\,.
\end{thm}

This theorem is proven under the assumption that the polynomial operators of vector fields are with  $ C^\infty $ coefficients, a hypothesis that will not be satisfied in our case.  The case with weakly
 regular coefficients appears less often in the literature.
We note, for example, that the case of the operator of the following form:
$$ \sum_{i,j=1}^{n}\,a_{i,j}\,X_i\,X_j +X_0 \,,$$
with non-regular coefficients $ a_{i, j} $ was studied by C.J.~Xu in \cite{Xu} and M.~Bramanti and L.~Brandolini in \cite{brandolini2000lp}, who prove  results of regularity as in H\"older and Sobolev spaces. Readers are referred to the article \cite{holder} for the study of H\"older regularity for the particular case of Kolmogorov operators with H\"older coefficients. None of these theorems apply directly. Moreover, they require a hypothesis of H\"older regularity which will not made here.
\subsection{Proof of the Sobolev regularity}
To prove Theorem \ref{prop 2}, we will show the Sobolev regularity associated to the following problem
$$
 K^*f=g \mbox{  with } f,g\in L^2_{loc}(\mathbb{R}^{2d})\,,$$ where $K^*$ is the formal adjoint of $ K $:
\begin{align}
 K^*=-v\cdot\nabla_x-\Delta_v+(v\wedge B_e)\cdot \nabla_v +v^2/4-d/2\,.\label{def_K*}
 \end{align}
The result of Sobolev regularity is the following:
\begin{thm}\label{thm_faible}
Let $d=2$ or $3$. We suppose that $B_e\in L^{\infty} (\mathbb{R}^{d}, \mathbb{R}^{d(d-1)/2})$. Then for all $f\in L^{2}_{loc} (\mathbb{R}^{2d})$, such that $ K^* f=g$ with $g\in L^{2}_{loc} (\mathbb{R}^{2d})$, then $ f\in \mathcal{H}^{2}_{loc}(\mathbb{R}^{2d})$\,. 
\end{thm}
Before proving Theorem \ref{thm_faible}, it is important to reduce our problem  to a problem with regular coefficients, in order to prove partial regularity in $ v $ for the following family of operators:
\begin{prop} \label{lem1}
Let $c_j\in L^{\infty}(\mathbb{R}^{d},L^{2}_{loc}(\mathbb{R}^d)),\,\forall j=1,...,d$ such that
\begin{align}\label{cond1}
\partial_{v_j}\,(c_j(x,v))=0\,\text{ in } \mathcal{D}'(\mathbb{R}^{2d})\,,\forall j=1,..,d\,.
\end{align}
Let $P_0$ be the Kolmogorov operator
\begin{align}\label{def:P0}
P_0:= -v\cdot \nabla_x -\Delta_v\,.
\end{align}
 If $h \in L^{2}_{loc}(\mathbb{R}^{2d})$ satisfies
\begin{equation}\label{prob1}
\begin{cases}
P_0h= \sum\limits_{j=1}^{d} c_{j}(x,v)\,\partial_{v_j}\,h_j+\tilde{g}\\
h_j, \tilde{g}\in L^{2}_{loc}(\mathbb{R}^{2d}),\,\forall j=1,...,d \,,
\end{cases}
\end{equation}
 then $ \nabla_v \, h \in L^{2}_{loc} (\mathbb{R}^{2d},\mathbb R^d)\,.$
\end{prop}
\begin{proof}
Let $ h $ satisfy \eqref{prob1}. We can work near a point $ (x_0, v_0) \in \Bbb R^{2d} $ {\it i.e.} in the ball $B((x_0,v_0),r_0)$ for some $r_0>0$. We show that 
 $ \nabla_v \, h$ belongs to  $L^2 (B((x_0,v_0),r_0/2),\mathbb R^{d})$\,.\\

\textbf{Step 1. Regular solution with compact support.}\\
We begin by assuming that $ h \in H^{2} (\mathbb{R}^{2d}) $ is supported in  $ B ((x_0, v_0), r_0) $ with $ r_0> 0 $ and we will establish a priori estimates.   We multiply Equation \eqref{prob1} by $ h $ and integrate it with respect to $ x $ and $ v $, we obtain
\begin{align}\label{eg_principale}
\langle P_0 h,h\rangle &=\sum\limits_{j=1}^{d}\langle  c_j\,\partial_{v_j}\,h_j, h\rangle +\langle \tilde{g}, h\rangle\,,
\end{align}
where $ \langle. \,, \,. \rangle $ denotes the Hilbertian scalar product on the real Hilbert space $ L^{2} (\mathbb{R}^{2d}) $.
Let's start by calculating the left side of the previous equality
\begin{align*}
\langle P_0 h,h\rangle &=\langle (v\cdot \nabla_{x}-\Delta_v)\,h,h\rangle \\
&=\int_{\mathbb{R}^{2d}}\, \left(v\cdot \nabla_x \,h \right)\, h\,dx\,dv-\int_{\mathbb{R}^{2d}}\,\left(\Delta_v\,h\right)\, h\,dx\,dv\,.
\end{align*}
Then performing an integrations by parts with respect to $ x $ and $ v $ in the previous equality, we get
\begin{align*}
  \langle P_0 h,h\rangle  &=\Vert \nabla_v \,h\Vert^{2}_{L^{2}(\mathbb{R}^{2d}, \Bbb R^{d})}\,,
\end{align*}
we used that the operator $ v \cdot \nabla_x $ is formally skew-adjoint in $ L^2 (\Bbb R^{2d}) $.

We now estimate each term in the right hand side of the equality \eqref{eg_principale}.  Using  Assumption  \eqref{cond1} and performing integration by parts with respect to $ v_j $, we have
\begin{align*}
\sum\limits_{j=1}^{d}\langle  c_j\,\partial_{v_j}\,h_j, h\rangle=-\sum\limits_{j=1}^{d}\langle  c_j \,h_j, \partial_{v_j}\,h\rangle\,.
\end{align*}
Then applying the Cauchy-Schwarz inequality to the scalar product, we get
\begin{align*}
\left\vert \sum\limits_{j=1}^{d}\langle  c_j (x,v)\,\partial_{v_j}\,h_j, h\rangle \right\vert \leq \sum\limits_{j=1}^{d} \Vert c_{j}\,h_j\Vert_{L^{2}(\mathbb{R}^{2d})}\,\Vert \partial_{v_j}\,h\Vert_{L^{2}(\mathbb{R}^{2d})}\,.
\end{align*}
Then we obtain that for all $ \eta> 0 $, there exists $ C_{\eta}> 0 $ such that
\begin{align*}
\left\vert \sum\limits_{j=1}^{d}\langle  c_j (x,v)\,\partial_{v_j}\,h_j, h\rangle \right\vert \leq \eta
\, \Vert \nabla_{v} h\Vert^{2}_{L^{2}(\mathbb{R}^{2d},\Bbb R^{d})} + C_{\eta}\sum\limits_{j=1}^{d}\, \Vert c_j\,h_j\Vert^{2}_{L^{2}(\mathbb{R}^{2d})}\,.
\end{align*}
Similarly, we get
\begin{align*}
\vert \langle \tilde{g}, h \rangle \vert \leq \frac{1}{2}\Vert \tilde{g}\Vert_{L^{2}(\mathbb{R}^{2d})}^2 +\frac{1}{2}\Vert h\Vert^{2}_{L^{2}(\mathbb{R}^{2d})}\,.
\end{align*}
By choosing $ \eta <1 $, we obtain the existence of a constant $ C> 0 $ such that for any $h$ and $h_j$ satisfy the problem \eqref{prob1}
\begin{align}\label{estim1}
\Vert \nabla_{v}\,h\Vert_{L^{2}(\mathbb{R}^{2d}, \Bbb R^{d})}^2 \leq C\,\left(\sum_{j=1}^{d}\,\Vert c_j\,h_j\Vert^{2}_{L^{2}(\mathbb{R}^{2d})} + \Vert \tilde{g}\Vert_{L^{2}(\mathbb{R}^{2d})}^2 +\Vert h\Vert_{L^{2}(\mathbb{R}^{2d})}^2  \right)\,.
\end{align}

\textbf{ Step 2. General solution.}

 \item[ $\mathrm{I}$.] \textbf{General framework and useful remark.} \\
 To have the regularity stated in Theorem \ref{lem1} for a solution $ h \in L^{2}_{loc} (\mathbb{R}^{2d}) $ of  \eqref{prob1}, 
  we introduce two cut-off functions $ \chi_1 $ and $ \chi_2 $ such that
\begin{equation*}\left\{
\begin{array}{l}
 \chi_1, \chi_2 \in C^{\infty}_0 (\mathbb{R}^{2d})\,,\quad\mathrm{Supp}\,\chi_1 \subset B((x_0,v_0),r_0)\,\\
 \chi_1=1  \text{ in } B((x_0,v_0),\frac{r_0}{2}) \text{ and } \chi_1=0 \text{ in } \mathbb{R}^{2d}\setminus B((x_0,v_0),r_0)\\
 \chi_1\,\chi_2=\chi_1\,.
 \end{array}\right.
 \end{equation*}
 Where $ r_0> 0 $ is fixed at the beginning of the proof.
 
 Let $ h_\delta $ be the function defined on $\mathbb R^{2d}$ by 
 \begin{align}
 h_\delta (x,v)=\chi_2\, (1-\delta^2\,\Delta_{x,v})^{-1}\, \chi_1\,h \,,\, \text{ for } \delta >0\,,
\end{align}
where the operator $ (1- \delta^2 \, \Delta_{x, v})^{- 1} $ is defined via the Fourier transform.

Note that $ h_\delta \in H^{2} (\mathbb{R}^{2d}) $ is supported in $ B ((x_0, v_0), r_0) $. Using the dominated convergence theorem one can show
\begin{equation}\label{eq:conv}
 h_\delta \to \chi_1\,h \mbox{ in } L^{2}(\mathbb{R}^{2d}) \mbox{ when } \delta \to 0\,. 
 \end{equation}
 We will now show that there is a constant $ C $ and $ \delta_0 $ such that for all $ \delta \in (0,\delta_0]$
 \begin{equation}\label{eq:bnd}
 \Vert \nabla_v\, h_\delta\Vert_{L^2(\Bbb R^{2d},\Bbb R^{d})} \leq C\,.
 \end{equation}
  Before giving the proof, we make the following simple  remark: 
\begin{remarque}  
   If $ Q_1 $ and $ Q_2 $ are two differential operators of order $ k_1 $ and $ k_2 $ respectively with coefficients in $ C ^{\infty}_0 $ such that $$ k_1+k_2\leq k\leq 2\,, $$ 
 then there exists a constant $C>0$ such that for all $\delta >0$
\begin{equation}\label{eq:remutile}
 \Vert Q_1\,(1-\delta^2\,\Delta_{x,v})^{-1} \, Q_2\,u\Vert_{L^{2}(\mathbb{R}^{2d})}\leq C\,\delta^{-k}\Vert u\Vert_{L^{2}(\mathbb{R}^{2d})}, \, \forall u\in L^{2}(\mathbb{R}^{2d})\,.
 \end{equation}
 \end{remarque}
 \item[ $\mathrm{II}$.] \textbf{Reformulation and relations of commutators.}\\
 Let's go back to the proof of inequality \eqref{eq:bnd}. We recall that $ h $ verifies 
 \begin{align}\label{eg1}
P_0 \,h=\sum\limits_{j=1}^{d} \,c_{j}(x,v)\, \partial_{v_j}\,h_j+\tilde{g}\,.
\end{align}
The goal is to follow the approach of the first step for $ h_\delta $ with however some differences. From \eqref{eg1}, let's look for the equation verified by $ h_\delta $. We have
\begin{align}\label{eg0}
P_0\,h_\delta =\chi_2\,(1-\delta^2\Delta_{x,v})^{-1}\,\chi_1\,P_0\,h +[\chi_2\,(1-\delta^2\Delta_{x,v})^{-1}\,\chi_1, P_0]\,h.
\end{align}
To control the commutator in the right-hand side, we will use the following relations of commutators:
\begin{align}
& [\varphi, P_0]=-\Delta_{v}\,\varphi +v\cdot\nabla_x \,\varphi  -2\nabla_{v}\varphi\cdot \nabla_v\,,\forall \varphi\in C^{\infty}_0(\mathbb{R}^{2d})\,, \label{com1}\\
&[\varphi, \partial_{v_j}]=-\partial_{v_j}\,\varphi\,, \,\,\forall \varphi \in C_{0}^{\infty}(\mathbb{R}^{2d})\,,\forall j=1,..,d\,,\label{com0}\\
 &[(1-\delta^2\,\Delta_{x,v}), P_0]=\delta^2\,[\Delta_{x,v}, v\cdot \nabla_{x}]=2\,\delta^2\,\sum_{j=1}^{d}\,\partial_{v_j}\partial_{x_j}\,,\label{com2}\\
  &[(1-\delta^2\,\Delta_{x,v})^{-1}, P_0]=-(1-\delta^2\,\Delta_{x,v})^{-1} \,[(1-\delta^2\,\Delta_{x,v}), P_0]\,(1-\delta^2\,\Delta_{x,v})^{-1}.\label{com3}
\end{align}
Using the relations of the previous commutators and equalities \eqref{eg1} and \eqref{eg0}, we rewrite $ P_0 \, h_\delta $ in the following form:
\begin{align}\label{eg2}
P_0 \,h_\delta &=\sum_{j=1}^{d}\,\chi_2\,(1-\delta^2\,\Delta_{x,v})^{-1}\,\chi_1\, c_j\, \partial_{v_j}\,h_j+B_0\,\tilde{g}\notag\\
&+B_1\,h+2\,\chi_2\,(1-\delta^2\,\Delta_{x,v})^{-1}\,\nabla_v\,\chi_1\cdot\nabla_v\,h \\
&+B_2\,h+2\,\nabla_v\chi_2\cdot \nabla_v\,\left( (1-\delta^2\,\Delta_{x,v})^{-1} \,\chi_1\,h \right)\notag\,.
\end{align}
Here the $ B_j $ for $ j = 0,1,2 $ are the uniformly bounded operators with respect to $ \delta $ on $ L^2 (\Bbb R^{2d}) $ (according to \eqref{eq:remutile})\,.
\begin{align}
 B_0&=\chi_2\,(1-\delta^2\,\Delta_{x,v})^{-1}\,\chi_1\,,\label{b0}\\
 B_1&=\chi_2\,(1-\delta^2\,\Delta_{x,v})^{-1}\,(\Delta_{x,v}\chi_1)+\chi_2\,(1-\delta^2\,\Delta_{x,v})^{-1}\,(v\cdot\nabla_x\chi_1)\,,\label{b1}\\
 B_2&=\chi_2\,(1-\delta^2\,\Delta_{x,v})^{-1}\,\delta^2\,R\,(1-\delta^2\,\Delta_{x,v})^{-1}\,\chi_1\,,\label{b2}
  \end{align}
  where $R$ is the second-order operator
  \begin{align}\label{R}
R=[\Delta_{x,v}, v\cdot\nabla_x]=2\sum_{j=1}^{d}\,\partial_{v_j}\partial_{x_j}\,.
\end{align}
\item[ $\mathrm{III}$.] \textbf{Proof of the uniform boundedness of $ \Vert \nabla_v \, h_\delta \Vert_{L^2(\Bbb R^{2d})} $.}\\
In order to show  \eqref{eq:bnd}, we multiply  equation \eqref{eg2} by $ h_\delta $ and integrate, we obtain
 \begin{align}\label{eg2:bis}
\langle P_0 \,h_\delta, h_\delta \rangle&=\sum_{j=1}^{d}\,\langle\chi_2\,(1-\delta^2\,\Delta_{x,v})^{-1}\,\chi_1\, c_j\, \partial_{v_j}\,h_j, h_\delta\rangle+\langle B_0\,\tilde{g}, h_\delta\rangle\notag\\
&+\langle B_1\,h, h_\delta\rangle+\langle 2\,\chi_2\,(1-\delta^2\,\Delta_{x,v})^{-1}\,\nabla_v\,\chi_1\cdot\nabla_v\,h, h_\delta\rangle \\
&+ \langle B_2\,h, h_\delta\rangle+\langle 2\,\nabla_v\chi_2\cdot \nabla_v\,\left( (1-\delta^2\,\Delta_{x,v})^{-1} \,\chi_1\,h \right), h_\delta\rangle\notag\,,
\end{align}
and we estimate term by term of the right-hand side of the previous equality.

For the first term, we use relation \eqref{com0} to switch $ \partial_{v_j} $ and $  \chi_k $ for $ k = 1,2 $, and obtain:
\begin{align*}
&\langle \chi_2\,(1-\delta^2\,\Delta_{x,v})^{-1}\,\chi_1\, c_j\, \partial_{v_j}\,h_j, h_\delta\rangle=I_0+I_1+I_2\,,
\end{align*}
where the $ I_j $ for $j=0,1,2$ are defined by
\begin{align}
I_0&=\langle \partial_{v_j}\left( \chi_2\,(1-\delta^2\,\Delta_{x,v})^{-1}\,\chi_1\,c_j\,h_j \right), h_\delta\rangle\,,\\
I_1&=\langle  \chi_2\,(1-\delta^2\,\Delta_{x,v})^{-1}\,(\partial_{v_j}\chi_1)\,c_j\,h_j, h_\delta\rangle\,,\\
I_2&=\langle (\partial_{v_j}\,\chi_2)\,(1-\delta^2\,\Delta_{x,v})^{-1}\,\chi_1 \,c_j\,h_j, h_\delta\rangle\,.
\end{align}
To estimate $ I_0 $, we integrate by parts with respect to $ v_j $
\begin{align*}
I_0 =&-\langle \chi_2\,(1-\delta^2\,\Delta_{x,v})^{-1}\,\chi_1\,c_j\,h_j ,\partial_{v_j}\, h_\delta\rangle\,,
\end{align*}
and then,  applying the Cauchy-Schwarz inequality, we get
\begin{align*}
\vert I_0\vert&\leq \Vert B_0\,(c_j\,h_j)\Vert_{L^{2}(\mathbb{R}^{2d})}\,\Vert \partial_{v_j}\,h_\delta\Vert_{L^{2}(\mathbb{R}^{2d})}\,.
\end{align*}
Then  using the uniform boundedness of  $ B_0 $ we obtain that for all $ \eta'> 0 $, there exists $ C_{\eta'}> 0 $ such that
\begin{align}\label{est:I0}
\vert I_0\vert \leq C_{\eta'}\,\Vert c_j\,h_j\Vert_{L^2(B((x_0,v_0),r_0))}^{2} +\eta'\,\Vert \partial_{v_j}\,h_\delta \Vert_{L^2(\Bbb R^{2d})}^2\,.
\end{align}
The two terms $ \vert I_1 \vert $ and $ \vert I_2 \vert $ satisfy the following estimates:
\begin{align}
\vert I_1\vert &\leq \Vert \chi_2\,(1-\delta^2\,\Delta_{x,v})^{-1}\,(\partial_{v_j}\chi_1)\,c_j\,h_j\Vert_{L^{2}(\mathbb{R}^{2d})}\,\Vert h_\delta\Vert_{L^{2}(\mathbb{R}^{2d})}\notag\\&\leq C\,\Vert c_j\,h_j\Vert_{L^2(B((x_0,v_0),r_0))}\,\Vert h_\delta \Vert_{L^{2}(\mathbb{R}^{2d})}\label{est:I1}\\
\vert I_2\vert &\leq \Vert(\partial_{v_j}\,\chi_2)\,(1-\delta^2\,\Delta_{x,v})^{-1}\,\chi_1 \,c_j\,h_j\Vert_{L^{2}(\mathbb{R}^{2d})}\,\Vert h_\delta\Vert_{L^{2}(\mathbb{R}^{2d})}\notag\\
&\leq C\,\Vert c_j\,h_j \Vert_{L^2(B((x_0,v_0),r_0))}\,\Vert h_\delta \Vert_{L^{2}(\mathbb{R}^{2d})}\,.\label{est:I2}
\end{align}
Using \eqref{est:I0}-\eqref{est:I2}, we get that for all $ \eta'> 0 $, there exists $ C_{\eta'}> 0 $ such that
\begin{align}\label{ine1}
&\left\vert  \sum_{j=1}^{d}\, \langle \chi_2\,(1-\delta^2\,\Delta_{x,v})^{-1}\,\chi_1\, c_j\, \partial_{v_j}\,h_j, h_\delta\rangle\right\vert \\&\leq C_{\eta'}\,\sum_{j=1}^{d} \,\Vert c_j\,h_j\Vert_{L^2(B((x_0,v_0),r_0))}^2+\eta'\,\Vert \nabla_{v}\,h_\delta\Vert_{L^{2}(\mathbb{R}^{2d},\Bbb R^{2d})}^2
+C_{\eta'}\,\Vert h_\delta \Vert_{L^{2}(\mathbb{R}^{2d})}^{2}\,,\notag
\end{align}
We show for the fourth and the sixth term of \eqref{eg2:bis} that for all $ \eta'' $ and $ \eta'''> 0 $, there exists $ C_{\eta''} $ and $ C_{\eta'''}> 0 $ such that
\begin{align}
\langle 2\,\chi_2\,(1-\delta^2\,\Delta_{x,v})^{-1}\,\nabla_v\,\chi_1\cdot\nabla_v\,h , h_\delta\rangle &\leq C_{\eta''}\Vert h\Vert_{L^2(B((x_0,v_0),r_0))}^2+\eta''\,\Vert \nabla_{v}\,h_\delta\Vert_{L^{2}(\mathbb{R}^{2d})}^2 \label{ine5}
\\
&\qquad +C_{\eta''}\,\Vert h_\delta\Vert_{L^{2}(\mathbb{R}^{2d})}^2 \notag \\
\langle 2\,\nabla_v\chi_2\cdot \nabla_v\,\left( (1-\delta^2\,\Delta_{x,v})^{-1} \,\chi_1\,h \right) , h_\delta\rangle &\leq C_{\eta'''}\Vert h\Vert_{L^2(B((x_0,v_0),r_0))}^2+\eta'''\,\Vert \nabla_{v}\,h_\delta\Vert_{L^{2}(\mathbb{R}^{2d},\Bbb R^{2d})}^2 \label{ine4}\\
&\qquad +C_{\eta'''}\,\Vert h_\delta\Vert_{L^{2}(\mathbb{R}^{2d})}^2\,.\notag
\end{align}

It remains to estimate the three remaining terms corresponding to the following scalar products:
\begin{align*}
\vert \langle B_0 \,\tilde{g}, h_\delta\rangle\vert  &\leq \Vert B_0\,\tilde{g}\Vert_{L^{2}(\mathbb{R}^{2d})}\,\Vert h_\delta\Vert_{L^{2}(\mathbb{R}^{2d})}\,,\\
\vert \langle B_1\,h, h_\delta \rangle \vert &\leq \Vert B_1\,h\Vert_{L^{2}(\mathbb{R}^{2d})}\,\Vert h_\delta\Vert_{L^{2}(\mathbb{R}^{2d})}\,,\\
\vert \langle B_2\,h, h_\delta \rangle \vert &\leq \Vert B_2\,h\Vert_{L^{2}(\mathbb{R}^{2d})}\,\Vert h_\delta\Vert_{L^{2}(\mathbb{R}^{2d})}\,.
\end{align*}
 Applying \eqref{eq:remutile}, we obtain the existence of constants $ C_{1}> 0 $ and $ C_2> 0 $ such that
\begin{align}
&\langle B_0 \,\tilde{g}, h_\delta\rangle \leq C_1\,\Vert \tilde{g}\Vert_{L^2(B((x_0,v_0),r_0))}\,\Vert h_\delta\Vert_{L^{2}(\mathbb{R}^{2d})}\,,\label{ine2bis}\\
&\langle B_1\,h, h_\delta \rangle \leq C_2\, \Vert h\Vert_{L^2(B((x_0,v_0),r_0))}\,\Vert h_\delta\Vert_{L^{2}(\mathbb{R}^{2d})}\,.\label{ine2}
\end{align}
To estimate the norm $\Vert B_2\,h\Vert_{L^{2}(\mathbb{R}^{2d})}$,   we apply \eqref{eq:remutile}  (with $Q_1=(1-\delta^2\,\Delta_{x,v})^{-1}$ and $Q_2=\delta^2 \,R$ where $R$ is defined in \eqref{R}).
We also obtain the existence of $ C_3 $ and $ \delta_0> 0 $ such that for all $ \delta \in (0, \delta_0] $
\begin{align}
\vert \langle B_2\,h, h_\delta \rangle \vert\leq C_3\, \Vert h\Vert_{L^2(B((x_0,v_0),r_0))}\,\Vert h_\delta\Vert_{L^{2}(\mathbb{R}^{2d})}\,.\label{ine3}
\end{align}
Finally, using the estimates \eqref{ine1}-\eqref{ine3}, and choosing $ \eta', \eta''$ and $ \eta'''> 0 $ such that $ \eta'+  \eta'' + \eta'''<1 $, we get the existence of a constant $ \tilde{C} $ and $ \delta_0> 0 $ such that $ \forall \delta \in (0, \delta_0] $
\begin{align}\label{ine4}
\Vert \nabla_{v}\,h_\delta\Vert_{L^{2}(\mathbb{R}^{2d},\Bbb R^{d})}^2 \leq \tilde{C}\,\left(\sum_{j=1}^{d}\,\Vert c_j\,h_j\Vert^{2}_{L^2(B((x_0,v_0),r_0))} + \Vert \tilde{g}\Vert_{L^2(B((x_0,v_0),r_0))}^2 +\Vert h\Vert_{L^2(B((x_0,v_0),r_0))}^2  \right)\,.
\end{align}
This completes the proof of \eqref{eq:bnd} \,.

\item[ $\mathrm{IV}$.] \textbf{Deduce the uniform local boundedness of $ \nabla_v \,h$.} \\
Let $ h \in L^{2} (\mathbb {R}^{2d}) $ satisfy \eqref{prob1}. On the one hand, according to \eqref{eq:conv} we know that
 $$ h_\delta \to \chi_1\,h \text{ in } L^2(\Bbb R^{2d}) \text{ when } \delta\to 0\,. $$
 This implies that
 $$ h_\delta \to \chi_1\,h \text{ and }  \nabla_v \,h_\delta \to \nabla_v (\chi_1\,h)  \text{ in } \mathcal{D}' (\Bbb R^{2d}) \text{ when } \delta \to 0 \,. $$
 On the other hand, according to inequality \eqref{ine4}, we obtain that $ (\nabla_v \, h_\delta)_{\delta} $ is bounded in $ L^2 (\mathbb{R}^{2d}) $ for $ \delta \in (0, \delta_0] $. By weak compacity, there is a subsequence $ (\delta_k)_{k \in \Bbb N} $ tending to $ 0 $ and a function $ u \in L^2 (\mathbb{R}^{2d}, \Bbb R^{d}) $ such that
 $$\nabla_v \,h_{\delta_k}\to u\, \text{ in } \,\mathcal{D}'(\Bbb R^{2d}) \, \text{ when } k\to+\infty\,.  $$
 Hence $ \nabla_v (\chi_1 \, h) = u $ in $ \mathcal{D}'(\Bbb R^{2d}) $. Then $ \nabla_v \, h \in L^2 (B((x_0,v_0),r_0/2), \Bbb R^{d}) $ \,.

\textbf{Step 3. Conclusion.}\\
Taking the information obtained near each point $ (x_0, v_0) $, we deduce that $ \nabla_v \, h \in L^{2}_{loc} (\Bbb R^{2d}, \Bbb R^{d}) $, which finishes the proof of Proposition \ref{lem1} \,.
\end{proof}
We now give the proof of Theorem \ref{thm_faible}.
\begin{proof}[Proof of Theorem \ref{thm_faible}]
 The key idea of the proof is to decompose the operator $K^*$ defined in  \eqref{def_K*} as follows:
\begin{align}\label{K_decom}
 K^*=P_0+(v\wedge B_e)\cdot\nabla_v+v^2/4-d/2\,.
\end{align}
Let $f\in L^{2}_{loc}(\mathbb{R}^{2d})$ and $B_e\in L^{\infty}(\mathbb{R}^{d},\mathbb{R}^{d(d-1)/2})$ such that
 $$ K^*\,f=g \text{ with }  g\in  L^{2}_{loc}(\mathbb{R}^{2d})\,,$$
 the goal is to show that $ f\in \mathcal{H}^{2}_{loc}(\mathbb{R}^{2d})$.
 
 \textbf{Step 1. Reformulation of the problem.}\\
 By following decomposition given in \eqref{K_decom} of the operator $ K^*$, we can consider our problem as a special case of the following generalized problem:
 \begin{align}\label{prob_refo}
\begin{cases}
P_0\,f= \sum\limits_{j=1}^{d} c_{j}(x,v)\,\partial_{v_j}\,h_j+\tilde{g}\\
h_j=f \in L^{2}_{loc}(\mathbb{R}^{2d})\,,\\
\tilde{g}=g-\frac{v^2}{4}\,f+\frac{d}{2}\,f\in L^{2}_{loc}(\mathbb{R}^{2d}),\,\forall j=1,...,d \,,
\end{cases}
\end{align}
 where $P_0$ the Kolomogrov operator defined in \eqref{def:P0} and
\begin{align*}
 c_j(x,v)=-(v\wedge\,B_e)_{j} \in L^{\infty}(\mathbb{R}^{d},L^{2}_{loc}(\mathbb{R}^{d}))\,,\forall j=1,...,d\,,
\end{align*}
and where we denote by $ (L)_{j} $ the jth component of the vector $ L \in \mathbb{R}^{d} $ \,. \\ Note that the coefficients $ c_j $ verify the condition \eqref{cond1} of Proposition \ref{lem1} because
$$ \partial_{v_j}(v\wedge B_e)_j=0\,,\forall j=1,..,d\,. $$
then Proposition \ref{lem1} gives $ \nabla_v \, f \in L^{2}_{loc} ( \mathbb{R}^{2d}, \Bbb R^{d}) $ \,.

\textbf{Step 2. Application to Theorem \ref{thm_transitoire}.} \\
Our operator $ P_0 $ can be written as follows:
\begin{align}\label{ecriture_2}
 P_0= -\left(\sum_{j=1}^{d} \,X_j^2 +X_0\right)\,,
 \end{align}
 where $ X_j $ for $ j = 0,1, .., d $ are defined by 
 \begin{align*}
X_j=\begin{cases}
\partial_{v_j}\, &\text{  if } \,j\neq 0 \,,\\
v\cdot \nabla_x\, &\text{  if  }\, j=0\,.
\end{cases}
\end{align*}
According to step $ 1 $, Problem \eqref{prob_refo} is a special case of the problem of regularity of type $ P_0 \, f = \tilde{h} $, with $ f $ and $ g \in L^{2}_{loc} (\mathbb{R}^{2d}) $ and $ \tilde{h} $ is given by
$$\tilde{h}=\sum\limits_{j=1}^{d} c_{j}(x,v)\,\partial_{v_j}\,f +g-\frac{v^2}{4}\,f+\frac{d}{2}\,f \,.$$
According to step 1, we have shown that $ \nabla_v \, f \in L^{2}_{loc} (\mathbb{R}^{2d},\Bbb R^{d}) $, which implies that $ \tilde{h} \in L^{2}_{loc} (\mathbb{R}^{2d}) $. Hence, by applying Theorem \ref{thm_transitoire},  we obtain $ f \in \mathcal{H}^{2}_{loc} (\mathbb{R}^{2d}) $, which completes the proof of Theorem \ref{thm_faible}\,.\\
\end{proof}
 \begin{remarque}\label{remark:utile}
In this part, we have shown a local Sobolev regularity on $\Bbb R^{2d}$, we  actually need it only in $\Omega \times \Bbb R^d$ where $\Omega$ is one of the open set of $\Bbb R^d$
 appearing when choosing local coordinates  for $\mathbb T^d$.
 \end{remarque}

\subsection{Proof of Theorem \ref{prop 2}}
Finally, we are ready to give the proof of Theorem~\ref{prop 2}. The accretivity of the operator $ K $ is clear.
To show that the operator is maximal, it suffices to show that there exists $ \lambda_{0}> 0 $ such that the operator $ T = K + \lambda_{0} \, Id $ is of dense image in $ L^2 (\mathbb{T}^{d}\times\mathbb{R}^{d}) $. As in \cite{nier2005hypoelliptic}, we take $  \lambda_{0} = \frac{d}{2}+1 \,$. \\
Let $f \in L^2(\mathbb{T}^{d}\times\mathbb{R}^{d})$ such that
\begin{align}
\langle u,(K + \lambda_{0} Id)\,w \rangle =0, \quad \forall w \in D(K)\,,
\label{6-z}
\end{align}
we have to show that $ u = 0 $.

For this we observe that equation \eqref{6-z} implies that 
$$  \left(K^*+\frac{d}{2} +1\right) u =0, \, \text{ in }\, \mathcal D'(\mathbb{T}^{d}\times\mathbb{R}^{d})\Longleftrightarrow K^*\,u=-(\frac{d}{2}+1)u\,,\, \text{ in }\, \mathcal D'(\mathbb{T}^{d}\times\mathbb{R}^{d})\,,$$
where $K^*$ is the operator defined in \eqref{def_K*}.

 Under the assumption that $ B_e\in L^\infty (\mathbb{T}^{d},\mathbb{R}^{d(d-1)/2}) $ and $ u \in D (K) \subset  L^2_{loc} (\mathbb{T}^{d}\times\mathbb{R}^{d}) $, Theorem \ref{thm_faible} and Remark \ref{remark:utile} show that $ f \in \mathcal{H}^2_{loc}(\mathbb{T}^{d}\times\mathbb{R}^{d})$. More precisely, using the compactness of $\mathbb T^d$, we have $\chi (v) f \in \mathcal{H}^2(\mathbb{T}^{d}\times\mathbb{R}^{d})$ for any $\chi\in C_0^\infty(\mathbb R^d)$. 
The rest of the proof is standard. The regularity obtained for $ f $ allows us to justify the integrations by parts and the cut-off argument given in \cite[Proposition 5.5]{nier2005hypoelliptic}. We note that in \cite{nier2005hypoelliptic}, a cut-off in $x$ and $v$ was necessary to develop the argument whereas here it suffices to perform a cut-off  in $v$. Here we refer to \cite[Proposition 3.1]{ZK}.

\section*{Acknowledgments} The author is grateful to Bernard Helffer for his important continued help and advice throughout the creation of this work  and  to Joe Viola for his useful remarks concerning both mathematical and grammatical questions of this paper. The author thanks also the Centre Henri Lebesgue ANR-11-LABX-0020-01 and the Faculty of Sciences (Section I) of the Lebanese University for their support.




\begin{thebibliography}{1}


\bibitem{beals1976operateurs}
R~Beals.
\newblock Op{\'e}rateurs invariants hypoelliptiques sur un groupe de lie
  nilpotent.
\newblock {\em S{\'e}minaire Goulaouic-Schwartz 1976/1977: {\'E}quations aux
  d{\'e}riv{\'e}es partielles et analyse fonctionnelle}, pages 1--8, 1976.


\bibitem{bolley1982condition}
P.~Bolley, J.~Camus, and J.~Nourrigat.
\newblock La condition de H{\"o}rmander-Kohn pour les op\'erateurs
  pseudo-differentiels.
\newblock {\em Communications in Partial Differential Equations},
  7(2):197--221, 1982.

\bibitem{brandolini2000lp}
M.~Bramanti and L.~Brandolini.
\newblock $L^p$ estimates for uniformly hypoelliptic operators with discontinuous
  coefficients on homogeneous groups.
\newblock {\em Rend. Sem. Mat. Univ. Politec. Torino}, 58(4):389--433, 2000. 

\bibitem{fokker1914ad}
 A.D.~Fokker.
\newblock {\em Ann. d. Physik}, 43:812, 1914.



\bibitem{nier2005hypoelliptic}
 B. ~ Helffer and F.~ Nier.
\newblock {\em Hypoelliptic estimates and spectral theory for Fokker-Planck
  operators and Witten Laplacians}.
\newblock Springer, 2005. Lecture Notes..

\bibitem{helffer1978hypoellipticite}
B~Helffer and J~Nourrigat.
\newblock Hypoellipticit{\'e} pour des groupes nilpotents de rang de nilpotence
  3.
\newblock {\em Communications in Partial Differential Equations},
  3(8):643--743, 1978.
  
  \bibitem{helffer1979caracterisation}
B.~Helffer and J.~Nourrigat.
\newblock  Caract\'erisation des op{\'e}rateurs hypoelliptiques homog{\`e}nes.
\newblock {\em Communications in Partial Differential Equations},
  4(8):899--958, 1979. 
  
  \bibitem{helffer1979approximation}
B. Helffer et J. Nourrigat.
\newblock Approximation d'un syst\`eme de champs de vecteurs et applications \`a l'hypoellipticit\'e.
\newblock {\em Arkiv f{\"o}r Matematik}, 17 (1-2), 1979.
  
  \bibitem{helffer1980hypoellipticite}
B.~Helffer and J.Nourrigat.
\newblock Hypoellipticit{\'e} maximale pour des op{\'e}rateurs polyn{\^o}mes de
  champs de vecteurs.
\newblock  Birkh\"auser {\em Progress in Mathematics}, 58, 1985.



\bibitem{herau2004isotropic}
F.~H{\'e}rau and F.~Nier.
\newblock Isotropic hypoellipticity and trend to equilibrium for the
  Fokker-Planck equation with a high-degree potential.
\newblock {\em Archive for Rational Mechanics and Analysis}, 171(2):151--218,
  2004.
  \bibitem{hormander1967hypoelliptic}
L.~H{\"o}rmander.
\newblock Hypoelliptic second order differential equations.
\newblock {\em Acta Mathematica}, 119(1):147--171, 1967.

  \bibitem{ZK} Z. Karaki.
   \newblock Trend to the equilibrium for the Fokker-Planck system with a strong external magnetic field.
   \newblock{\em hal-01975138}, 2019.
 
  
  \bibitem{kirillov1962unitary}
A.~ Kirillov.
\newblock Unitary representations of nilpotent lie groups.
\newblock {\em Russian mathematical surveys}, 17(4):53, 1962.


\bibitem{Planck-FPE-1917}
M.~Planck.
\newblock {\em Sitzber. Preuss. Akad. Wiss.}, page 324, 1917.

 

\bibitem{rockland1978hypoellipticity}
C.~Rockland.
\newblock Hypoellipticity on the  Heisenberg group-representation-theoretic
  criteria.
\newblock {\em Transactions of the American Mathematical Society}, 240:1--52,
  1978.
  
  \bibitem{rothschild1979criterion}
L.P. Rothschild.
\newblock A criterion for hypoellipticity of operators constructed from vector
  fields.
\newblock {\em Communications in Partial Differential Equations},
  4(6):645--699, 1979.

\bibitem{rothschild1976hypoelliptic}
L.~P. Rothschild and E.M.~ Stein.
\newblock Hypoelliptic differential operators and nilpotent groups.
\newblock {\em Acta Mathematica}, 137(1):247--320, 1976.

\bibitem{holder}
A.~Pascucci.
\newblock H{\"o}lder regularity for a  Kolmogorov equation.
\newblock {\em Transactions of the American Mathematical Society}, pages
  901--924, 2003.
  
  \bibitem{Xu}
C.J.~Xu.
\newblock Regularity for quasilinear second-order subelliptic equations.
\newblock {\em Comm. Pure Appl. Math.}, 45(1):77--96, 1992.


\end{thebibliography}
\end{document}